\documentclass{amsart}
\usepackage{amssymb,amsmath,graphicx}

\usepackage[cmtip,all]{xy}
\usepackage{amscd}
\usepackage{graphics}

\usepackage[usenames]{color}
\definecolor{Red}{rgb}{1,0,0}

\newcounter{mycounter}

\newtheorem{theorem}{Theorem}

\theoremstyle{definition}
\newtheorem{definition}[theorem]{Definition}
\newtheorem{example}[theorem]{Example}

\theoremstyle{remark}
\newtheorem{remark}[theorem]{Remark}

\numberwithin{equation}{section}

\begin{document}

\title[Word calculus in the fundamental group of the Menger curve]{Word Calculus in the fundamental group \\ of the Menger curve}

\author{Hanspeter Fischer}

\address{Department of Mathematical Sciences\\Ball State University\\Muncie, IN 47306\\U.S.A.}

\email{fischer@math.bsu.edu}

\author{Andreas Zastrow}

\address{Institute of Mathematics, University of Gda\'nsk, ul. Wita Stwosza 57,
80-952 Gda\'nsk, Poland}

\email{zastrow@mat.ug.edu.pl}

\thanks{}

\subjclass[2000]{20F65; 20E08, 55Q52, 57M05, 55Q07}

\keywords{Generalized Cayley graph, Menger universal curve, word sequence}

\date{\today}

\commby{}

\begin{abstract}
The fundamental group of the Menger universal curve is uncountable and not free, although all of its finitely generated subgroups are free. It contains an isomorphic copy of the fundamental group of every one-dimensional separable metric space and an isomorphic copy of the fundamental group of every planar Peano continuum.
 We give an explicit and systematic combinatorial description of the fundamental group of the Menger universal curve and its generalized Cayley graph in terms of word sequences. The word calculus, which requires only two letters and their inverses, is based on Pasynkov's partial topological product representation and can be expressed in terms of a variation on the classical puzzle known as the Towers of Hanoi.
\end{abstract}

\mbox{\hspace{1pt}}\vspace{-12pt}

 \maketitle

\section{Introduction}

\noindent The fundamental group $\pi_1({\mathbb M})$ of the Menger universal curve $\mathbb{M}$ contains an isomorphic copy of the fundamental group of every one-dimensional separable metric space \cite[\S5]{CF2} and an isomorphic copy of the fundamental group of every planar Peano continuum \cite[\S6]{CC}. All of these groups have recently come under renewed scrutiny (see, e.g., \cite{CC3,CL,E}), while eluding concrete description by classical means.

 The uncountable group $\pi_1(\mathbb{M})$ is not free \cite[Thm~2.2]{CF}, but all of its finitely generated subgroups are free \cite{CF2}, suggesting that any object resembling a Cayley graph should be tree-like. Indeed, in \cite{FZ} we have shown that, given any one-dimensional path-connected compact metrizable space, there is, in principle, an $\mathbb{R}$-tree that can be regarded as a generalized Cayley graph for its fundamental group.

 In this article, we present a concrete and systematic construction of a generalized Cayley graph for the fundamental group of the Menger universal curve by way of sequences of finite words in two letters and their inverses.
 Our word calculus is based on a representation of $\mathbb{M}$ as the limit of an inverse sequence of certain finite 4-valent graphs $X_n$ (see \S\ref{sp} and \S\ref{edgelabels} below), constituting a special case of Pasynkov's classical construction  \cite{P}. We choose these graphs so that each $X_n$ can be reinterpreted as the directed state graph of a variation on the classical Towers of Hanoi puzzle with $n+1$ disks (see \S\ref{TH}). In this way, we obtain a ``mechanical'' description of $\pi_1(\mathbb{M})$.

The material is organized in such a way that the sections on the word calculus for $\pi_1(\mathbb{M})$ and its generalized Cayley graph (see \S\ref{map}--\ref{CayM}) can be read independently of the game interpretation. In fact, most aspects are presented both in terms of word sequences and in terms of disk movements. ({\sl References to the game are set in slanted typeface}.) However, we hope that the reader who chooses to engage in the puzzle correspondence will gain better visual insight into the overall combinatorics, which are ultimately driven by the injection $\pi_1(\mathbb{M}) \hookrightarrow \check{\pi}_1(\mathbb{M})$ into the first \v{C}ech homotopy group (see \S\ref{cech}).
We close with a brief discussion of a concrete embedding of the much-studied Hawaiian Earring into our model (see~\S\ref{HEemb}).

\subsection{About the word calculus} As we shall see, each graph $X_n$ naturally covers a bouquet of two directed circles $x$ and $y$, allowing us to describe the set $\Omega_n$ of based edge-loops in $X_n$ in terms of finite words over the alphabet $\{x^{+1}, x^{-1}, y^{+1},y^{-1}\}$. The topological bonding maps $f_n:X_{n+1}\rightarrow X_n$ then induce combinatorial bonding functions $\phi_n:\Omega_{n+1}\rightarrow \Omega_n$.  The reduced edge-loops (reduced as words in the free group on $\{x,y\}$) are canonical representatives for the elements of $\pi_1(X_n)$ and form a reduced inverse sequence $(\Omega'_n,\phi_n')_{n\in\mathbb{N}}$ whose limit is isomorphic to $\check{\pi}_1(\mathbb{M})$. The combinatorial description of $\pi_1(\mathbb{M})$ then follows along the lines of \cite{FZ}: From the set of {\em reduced}-coherent word sequences in $(\Omega'_n,\phi_n')_{n\in\mathbb{N}}$, we select only those which have eventually stable {\em unreduced} projections at every fixed lower level. The stabilizations of this selection of sequences form a group under the natural binary operation of termwise concatenation, followed by reduction and restabilization, and this group is isomorphic to $\pi_1(\mathbb{M})$. The fact that all functions $\phi_n$ can be calculated by one and the same explicit formula (see Definition~\ref{phi} and Example~\ref{ExProj}) allows us to present a systematic algorithm which, when randomized, outputs a generic element of $\pi_1(\mathbb{M})$ (see \S\ref{algorithm}).

The generalized Cayley graph of $\pi_1(\mathbb{M})$ comprises an analogous (stabilized) selection of sequences of words describing edge-paths in $X_n$ that start at the base vertex, but do not necessarily end there, and it is turned into an $\mathbb{R}$-tree using a recursive word length calculation (see \S\ref{CayM}).  The group, which is  a distinguished subset of this $\mathbb{R}$-tree, acts freely on the $\mathbb{R}$-tree by homeomorphisms in the same natural way, with quotient homeomorphic to $\mathbb{M}$. Furthermore, arcs in the $\mathbb{R}$-tree which connect points corresponding to group elements naturally spell out word sequences that represent the group-theoretic difference of their endpoints. Considering the underlying limitations discussed in \cite[\S1]{FZ}, these features represent a best possible generalization of the concept of a classical Cayley graph for $\pi_1(\mathbb{M})$.

\vfill

\subsection{About the board game} {\sl In our version of the Towers of Hanoi, the placement of the disks is restricted to within the well-known unique shortest solution of the classical puzzle, while we allow for backtracking within this solution and for the turning over of any disk that is in transition. We color the disks white on one side and black on the other. Then the state graph of this new ``puzzle'' is isomorphic to $X_n$, with edges corresponding to situations where all disks are on the board and vertices marking the moments when disks are in transition. The exponents of the edge labels ($x^{\pm 1}$ or $y^{\pm 1}$) indicate progress (``$+1$'') or regress (``$-1$'') in solving the classical puzzle (we add a game reset move when the classical puzzle is solved) and their base letters indicate whether the two disks to be lifted at the respective vertices of this edge are of matching (``$x$'') or mismatching (``$y$'') color. Hence, each edge-path through $X_n$ corresponds to a specific evolution of this game, as recorded by an observer of the movements of the $n+1$ disks.

 Our word calculus can be modeled by aligning an entire sequence of such puzzles with incrementally more disks into an inverse system, whose bonding functions between individual games simply consist of ignoring the smallest disk. Subsequently, every combinatorial notion featured in the description of the generalized Cayley graph has a mechanical interpretation in terms of this sequence of puzzles (see \S\ref{CayM}).

In particular, the elements of $\pi_1(\mathbb{M})$ are given by those coherent sequences of plays which return all disks to the base state, while all along making sure that any seemingly cancelling interaction among finitely many disks turns out not to cancel when the movements of sufficiently many additional disks are revealed (see \S\ref{visualization}).}

\vfill

\section{A parametrization of Pasynkov's representation of $\mathbb{M}$}
\label{sp}

\noindent The Menger universal curve  is usually constructed as the intersection $\bigcap_{n=0}^\infty M_n$ \linebreak of a nested sequence of cubical complexes $M_0\supseteq M_1\supseteq M_2\supseteq \cdots$ in $\mathbb{R}^3$, where $M_0=[0,1]^3$  and $M_{n+1}$ equals the union of all subcubes of $M_n$ of the form \linebreak $\prod_{i=1}^3 [c_i/3^n,(c_i+1)/3^n]$  with $c_i\in \mathbb{Z}$ and $c_i \equiv 1$ (mod 3) for at most one $i$. (The universal curve derives its name from the fact that it contains a homeomorphic copy of every one-dimensional separable metric space \cite[Theorem~6.1]{MOT}.)

Following Pasynkov \cite{P}, we instead represent the Menger universal curve as an inverse limit of partial topological products over the circle $S^1$ with diadic fibers.
Specifically, we first parametrize the circle $S^1=\mathbb{R}/\mathbb{Z}$ by way of the additive group $[0,1)$ modulo 1, using binary expansion. We  then let $\{A,B\}$ be a discrete two-point space and define the space $X_0$ as the quotient of $S^1\times\{A,B\}$ obtained from identifying $(0,A)\sim (0,B)$ and $(1/2,A)\sim (1/2,B)$, that is, from identifying $(.0,A)\sim (.0,B)$ and $(.1,A)\sim (.1,B)$. For $t\in \{.0,.1\}$, we denote the point $\{(t,A),(t,B)\}$ of $X_0$ by $(t,\Box)$.  Recursively, for every positive integer $n$, we define the space $X_n$ as the quotient of $X_{n-1}\times\{A,B\}$ obtained from identifying $((t,w),A)\sim ((t,w),B)$ whenever $t=b/2^{n+1}$ for some odd integer $b$, that is, for every $t=.t_1t_2\cdots t_{n}1$ with $t_i\in \{0,1\}$; again, the point $\{((t,w),A),((t,w),B)\}$ of $X_n$ will be denoted by $((t,w),\Box)$. For convenience, we will write, for example, $(.0110,AB\Box B)$ instead of $((((.0110,A),B),\Box),B)$.
\begin{figure}[h]
\includegraphics{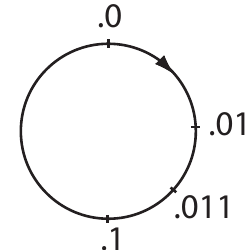} \includegraphics[scale=.5]{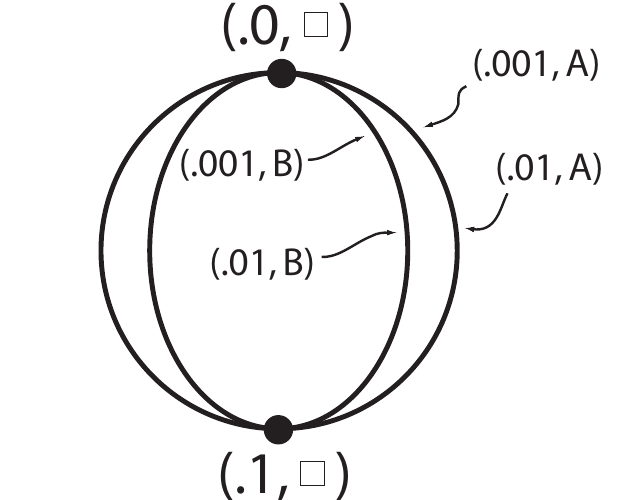} \includegraphics{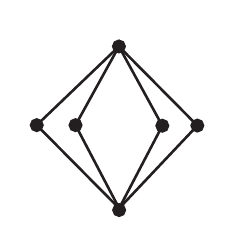}
\caption{\label{X0} $S^1$ (left),  $X_0$ (middle) and its subdivision $X_0^\ast$ (right).}
\end{figure}

\begin{figure}[h!]
\includegraphics{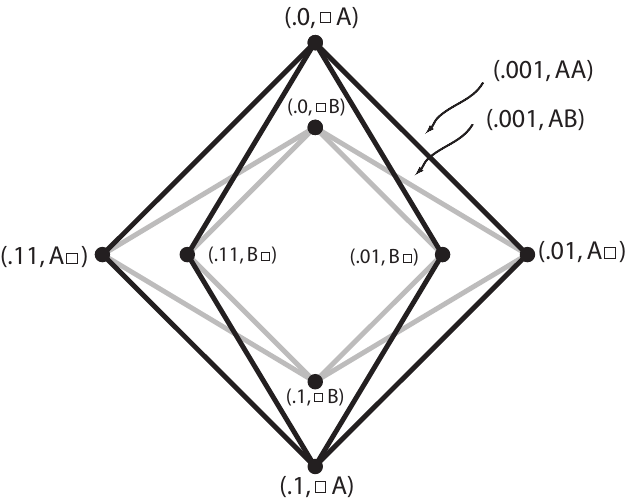}\hspace{-.1in}  \includegraphics{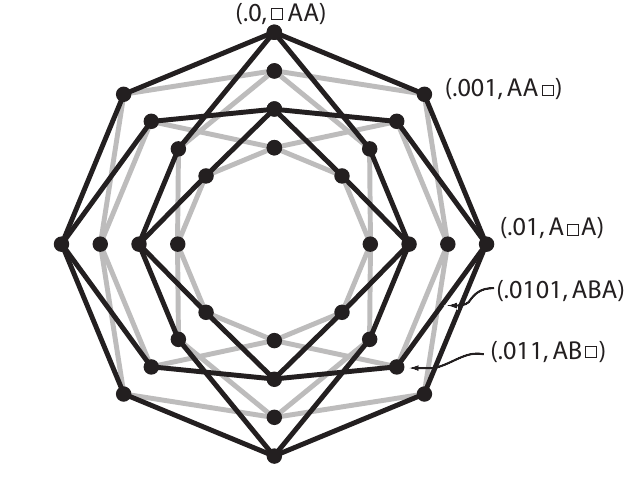}
\caption{\label{X1} The graphs $X_1$ (left) and $X_2$ (right). In this rendering, the vertices are arranged clockwise in the first coordinate and alphabetically in the second coordinate (from outside to inside).}
\end{figure}

We record the following observations for later reference (see Figures~\ref{X0} and \ref{X1}):
\begin{list}{(\ref{sp}.\arabic{mycounter})}{\usecounter{mycounter}}
\item {\bf Valence.}  Each $X_n$ is a 4-valent graph (and a simplicial complex for $n\geqslant 1$) with $2^{2n+1}$ vertices and $2^{2n+2}$ edges.

 \item  {\bf Basic building block.} The space $X_0$ can be regarded as the basic building block for the construction. It is a graph with two vertices, labeled $(.0,\Box)$ and $(.1,\Box)$, which are connected by four edges.

 \item {\bf Doubling.}\label{doubling} The graph $X_{n+1}$ can be obtained from $X_n$ by first forming the barycentric subdivision $X^\ast_n$ of $X_n$ and then, for every vertex $v$ of $X_n$, replacing the vertex star Star$(v,X^\ast_n)$ of $v$ in $X_n^\ast$ by its double along the boundary, i.e., replacing it by a copy of $X_0^\ast$, where the four edge-midpoints of $X_0$ assume the role of the four boundary points of Star$(v,X^\ast_n)$.

  \item {\bf Vertices.} \label{vertices} Every vertex of $X_n$ is of the form $(t,w)$ where $t=.t_1t_2\cdots t_{n+1}$ with $t_i\in\{0,1\}$ and $w$ is a word of length $n+1$ in the letters $A$ and $B$, except that $w$ contains the symbol $\Box$ in position $d(t)$ instead of a letter, where $d(t)$ is  the {\em bit length} of $t$:\begin{itemize}
          \item[(i)] if $t=0$, then $d(t)=1$; \item[(ii)] if $t\not=0$, then  $d(t)$ is the position of the last $1$ in the (terminating) binary expansion of $t$ (i.e. $t=b/2^{d(t)}$ for some odd integer $b$).\end{itemize}

  \item {\bf Edges.} \label{edges} Two vertices $(t,w)$ and $(s,v)$ of $X_n$ ($n\geqslant 1$) span an edge of $X_n$ if and only if $|t-s|=1/2^{n+1}$ and $w\cap v\not=\emptyset$, in which case $w\cap v$ is a word in the letters $A$ and $B$ of length $n+1$. (Here, we abuse the intersection notation by interpreting the word $AB\Box B$ as the set $\{ABAB,ABBB\}$: $AB\Box B\cap A\Box AB=ABAB$ and $AB\Box B\cap A\Box AA=\emptyset$.) The points on an edge between the vertices $(t,w)$ and $(s,v)$ are of the form $(r,w\cap v)$ with $r$ (on the short arc of $S^1$) between $t$ and $s$. In particular, the barycentric subdivision point of an edge between the vertices $(t,w)$ and $(s,v)$ is given by $((t+s)/2,w\cap v)$.

\item {\bf Bonding maps.} \label{bonding} There are natural maps $f_n:X_{n+1}\rightarrow X^\ast_n$, given by the formula $f_n(t,w)=(t,\sigma(w))$ where $\sigma(w)$ is the right-shift of the word  $w$ which ``forgets'' the last symbol.

\item {\bf Linearity.} \label{linearity} Each $f_n$ maps every edge of $X_{n+1}$ linearly onto an edge of $X^\ast_n$.

\item {\bf Preimages.} \label{preimages} The preimage $f_n^{-1}(e)$ in $X_{n+1}$ of an edge $e$ of $X_n$ is a figure-X whose four boundary vertices map to the endpoints of $e$ in pairs under $f_n$ and whose interior vertex maps to the midpoint of $e$.

\item {\bf Lifts.} \label{lifts} By composing the map $(t,w)\mapsto((t,w),A)$, which appends the letter $A$ to the word $w$, with the quotient map $X_n\times \{A,B\}\rightarrow X_{n+1}$, we obtain a  lift $g_n:X_n^\ast\hookrightarrow X_{n+1}$ with $f_n\circ g_n=id_{X_n}$. (See Figure~\ref{X1}, black edges.)

\end{list}

\begin{theorem}[Pasynkov \cite{P}] \label{Pasynkov}The limit $\displaystyle \mathbb{M}=\lim_{\longleftarrow} \left(X_1\stackrel{f_1}{\longleftarrow} X_2 \stackrel{f_2}{\longleftarrow} X_3 \stackrel{f_3}{\longleftarrow} \cdots\right)$ is homeomorphic to the Menger universal curve.
\end{theorem}

\begin{proof}[Sketch of Proof] Since each $X_n$ is one-dimensional, compact, connected and metrizable, so is $\mathbb{M}$ (see \cite[Corollary~8.1.7]{Pears} and \cite[Theorem~2.4]{Nadler}, for example).
 Moreover, it follows from (\ref{sp}.\ref{preimages}) that for every $n$, the preimage $p^{-1}_n(C)$ in $\mathbb{M}$ of every closed vertex star $C$ of $X_n$ under the inverse limit projection $p_n:\mathbb{M}\rightarrow X_n$ is connected. Hence, $\mathbb{M}$ has Property~S (every open cover can be refined by a finite cover of connected sets) and is therefore locally connected \cite[Theorem~IV.3.7]{W}.
 Also, it is not difficult to see that $\mathbb{M}$ does not have any local cut points, using the simple fact that every edge-path in $X_n$ can avoid any point of $X_n$ by a detour of at most 6 edges and considering the embedding of the approximating graphs $X_n$ into $\mathbb{M}$ induced by (\ref{sp}.\ref{lifts}).
 Finally, every nonempty open subset of $\mathbb{M}$ contains an embedded edge $e$ of some $X_n$ and hence
  it contains the embedded preimage of $e$ in $X_{n+3}$ under the map $f_{n+2}\circ f_{n+1}\circ f_n$, which in turn contains a subdivision of the (non-planar) complete bipartite graph $K_{3,3}$ by (\ref{sp}.\ref{preimages}). In summary, $\mathbb{M}$ satisfies Anderson's characterization  of the Menger universal curve \cite{A1,A2} (see also \cite[Theorem~4.11]{MOT}).
\end{proof}

\section{The \v{C}ech homotopy group of $\mathbb{M}$}\label{cech}

\noindent It is straightforward to describe the first \v{C}ech homotopy group $\check{\pi}_1(\mathbb{M})$ using the inverse system of the previous section. The fundamental group of $X_0$ is free on any three of the four oriented edges from $(.0,\Box)$ to $(.1,\Box)$, i.e., $\pi_1(X_0, (.0,\Box))=F_3$.
 The graph $X_{n+1}$ contains natural copies of $X^\ast_n$ by (\ref{sp}.\ref{lifts}). Indeed, by (\ref{sp}.\ref{doubling}) $X_{n+1}$ is easily seen to be homotopy equivalent to a graph obtained from $X_n$ by gluing one copy of $(X_0,(.0,\Box))$ with its base point to each of the vertices of $X_n$. Hence, we have $\pi_1(X_{n+1})=\pi_1(X_n)\ast F_3\ast F_3 \ast \cdots \ast F_3$, with one free factor of $F_3$ for each vertex of $X_n$. Moreover, the homomorphism  $f_{n\#}:\pi_1(X_{n+1})\rightarrow \pi_1(X_n)$ trivializes the additional $F_3$ factors by (\ref{sp}.\ref{bonding}).
If all we wanted to do was compute the first \v{C}ech homotopy group,
 this would be all we need to know:
 \[\check{\pi}_1(\mathbb{M})=\lim_{\longleftarrow}\big(F_3\ast (F_3\ast F_3)\longleftarrow F_3\ast F_3\ast F_3\ast(F_3\ast F_3\ast \cdots \ast F_3)\longleftarrow \cdots\big)\]
 However, since we are interested in describing  $\pi_1(\mathbb{M})$, which consists of homotopy classes of continuous loops,
  and since the free factors are attached along a countable dense subset, we need to use the approach taken in \cite{FZ} of combinatorially examining the natural injective homomorphism $\pi_1(\mathbb{M})\hookrightarrow \check{\pi}_1(\mathbb{M})$ from \cite{CF2}.

   The machinery of \cite{FZ} applies to the set-up here, because of Theorem~\ref{Pasynkov} and (\ref{sp}.\ref{linearity}) (cf. \cite[Lemma~3.1]{FZ}). As we shall see, it turns out to produce a very systematic and concrete description of $\pi_1(\mathbb{M})$ and its generalized Cayley graph, which can be mechanically illustrated using a variation on the classical Towers of Hanoi puzzle.

\section{Edge labels}\label{edgelabels}

\noindent The endpoints of a given edge $e$ of $X_n$ are of the form $(t,w)$ and $(t+1/2^{n+1},v)$, where  $w\cap v$ is a word in the letters $A$ and $B$ of length $n+1$, by (\ref{sp}.\ref{edges}). By (\ref{sp}.\ref{vertices}), the words $w$ and $v$ have the symbol ``$\Box$'' in position $d(t)$ and $d(t+1/2^{n+1})$, respectively. We label the edge $e$ according to the following rule:
\[label(e)=\left\{ \begin{array}{ccc}
                     x & \mbox{ if } & letter(d(t),w\cap v)=letter(d(t+1/2^{n+1}),w\cap v)\\
                     y &  \mbox{ if } & letter(d(t),w\cap v)\neq letter(d(t+1/2^{n+1}),w\cap v)
                   \end{array}
\right.\]
where $letter(d,w\cap v)=d^{th}$ letter of the word $w\cap v$.

\subsection{Directed edges}
If the edge $e$ above is labeled ``$x$'', respectively ``$y$'', we label the directed edge from $(t,w)$ to $(t+1/2^{n+1},v)$ by ``$x^{+1}$'', respectively ``$y^{+1}$''; and we label the directed edge
from $(t+1/2^{n+1},v)$ to $(t,w)$  by ``$x^{-1}$'', respectively ``$y^{-1}$''.

\subsection{Vertex neighbors}\label{vertexnb}
 The four vertex neighbors $(s,v)$ of a vertex $(t,w)$ in $X_n$ are readily found, given a letter  $a\in \{x^{+1},x^{-1},y^{+1},y^{-1}\}$: compute $s=t\pm 1/2^{n+1}$ according to the exponent of $a$ and form $v$ from $w$ by first swapping the symbol ``$\Box$'' with the letter in position $d(s)$, where $d(s)$ is the bit length of $s$, and then changing that letter to its opposite ($A\leftrightarrow B$) in case  $a=y^{\pm1}$. (See Figure~\ref{Star}.)
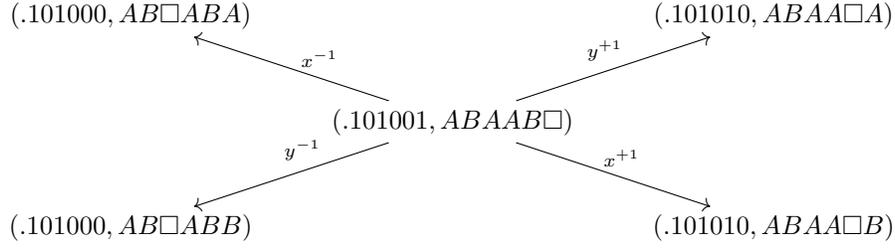
\begin{figure}[h]
\parbox{4.9in}{
\[\xymatrix{ (.101000,AB\Box ABA) & & (.101010,ABAA\Box A)\\
&  \ar[ul]_<(.3){x^{-1}} \ar[dl]_<(.3){y^{-1}} \ar[ur]^{y^{+1}} \ar[dr]^{x^{+1}} (.101001,ABAAB\Box ) & \\
(.101000,AB\Box ABB) & & (.101010, ABAA\Box B)}
\]}
\caption{\label{Star} The vertex star of $(.101001,ABAAB\Box )$ in $X_5$.}
\end{figure}

\subsection{Edge-paths}
Labeled as  above, $X_n$ is a covering space of a bouquet of two directed circles $x$ and $y$. We fix $x_n=(.0,\Box AA\cdots A)$ as the base vertex of $X_n$.
We can then record edge-paths starting at $x_n$ as finite words over the alphabet $\{x^{+1},x^{-1},y^{+1},y^{-1}\}$.
Since the covering is not regular if $n\geqslant 2$, the choice of basepoint matters when calculating whether an edge-path is a loop. (See  Section~\ref{loop}.) For example, the word $x^{+1}y^{-1}y^{-1}x^{+1}$ forms an edge-loop in $X_2$ based at $x_2=(.0,\Box AA)$, but the same word represents  an edge-path that ends at vertex $(.001,BB\Box)$ when we start at vertex $(.001,AA\Box)$. (See Figure~\ref{game}.)

\subsection{Reduced paths and cancellation}\label{cancel}
We  denote by $\mathcal W$ the set of all finite words over the alphabet  $\{x^{+1},x^{-1},y^{+1},y^{-1}\}$, including the empty word, and define
\[\Omega_n=\{\omega\in {\mathcal W}\mid \omega \mbox{ forms an edge-loop in } X_n \mbox{ based at } x_n\}.\]
For every word $\omega\in {\mathcal W}$, there is a unique reduced word $\omega'$, which results from cancelling any subwords of the form $x^{+1}x^{-1}$, $x^{-1}x^{+1}$, $y^{+1}y^{-1}$, and $y^{-1}y^{+1}$ (in any order) until this is no longer possible.
   For $ S\subseteq {\mathcal W}$, we define $S'=\{\omega'\mid \omega\in S\}$. Since $\omega'$ is a canonical representative for the
edge-path homotopy class of $\omega$, we have
$\pi_1(X_n,x_n)\cong \Omega_n'\leqslant {\mathcal W}'= F_2$, where $F_2$ denotes the free group on $\{x,y\}$.

\begin{remark}
The reader be advised that in \cite{FZ} all edge-paths are expressed in terms of visited {\em vertices}, rather than traversed {\em edges}.
\end{remark}

\section{A variation on the ``Towers of Hanoi'' puzzle}\label{TH}
{\sl

\noindent The classical puzzle known as the Towers of Hanoi consists of $n+1$ holed disks $\{1,2,\dots, n+1\}$, staggered in diameter, and a board with three pegs. The sole player stacks all disks, largest (Disk 1) to smallest (Disk $n+1$), on the leftmost peg (Peg 0) and tries to recreate this tower on the rightmost peg (Peg 2) by moving one disk at a time from peg to peg, using the middle peg (Peg 1) for temporary storage. What makes the game a puzzle is the rule that one is not allowed to ever place a larger disk on top of a smaller one.

\subsection{Shortest solution}\label{shortest} There is a well-known unique shortest solution to this puzzle through $2^{n+1}$ assemblies of the board with $2^{n+1}-1$ disk transitions \cite{Wa,CE}. Enumerating the assembly stages by binary fractions, we summarize it as follows:
 \begin{itemize}
 \item[I.] At assembly stage $t=.t_1t_2\cdots t_{n+1}$, with $t_i\in \{0,1\}$, Disk $i$ is on Peg $P_i(t)$, where $P_1(t)\equiv -t_1$ (mod 3) and  $P_i(t)\equiv P_{i-1}(t)+(-1)^i(t_i-t_{i-1})$ (mod 3).\linebreak \vspace{-10pt}

     \item[II.] Between the assembly stages $t-1/2^{n+1}$ and  $t$, Disk $d(t)$ moves from Peg $\left(P_{d(t)}(t)-(-1)^{d(t)}\right)$ (mod 3) to Peg $P_{d(t)}(t)$ (mod 3); with $d(t)$ as in (\ref{sp}.\ref{vertices}).
     \end{itemize}

\begin{remark}[{\bf Local solution rule, disk running directions, and leading disks}] {\sl While one popular (but not very practical) method of solution is by recursion, the above formulas dictate a simple local solution rule. Each disk has a fixed (cyclic) ``running direction''  toward completion of the puzzle: the largest disk moves left and the directions of the smaller disks alternate from there.
At any moment when all disks are on the board, except during the initial and the final assembly stages, there are exactly two movable disks: the smallest disk (movable in both directions) and one other disk (movable in only one direction). The local solution rule consists of always moving the largest possible disk in its natural running direction, thus moving the smallest disk every other time. To play in reverse, one would always move the largest possible disk against its natural running direction. In this way, given any intermediate assembly stage along the shortest solution, one of the two movable disks is designated to advance the solution (the ``leading disk'') and the other one to undo the previous move. (See Figure~\ref{placementA}.) For further reading, see \cite{H}.}
\end{remark}

\noindent {\bf General assumption:}\\
{\sl We will restrict the placement of the disks on the board to within this solution}.

\begin{figure}[h]

 \includegraphics[scale=.6]{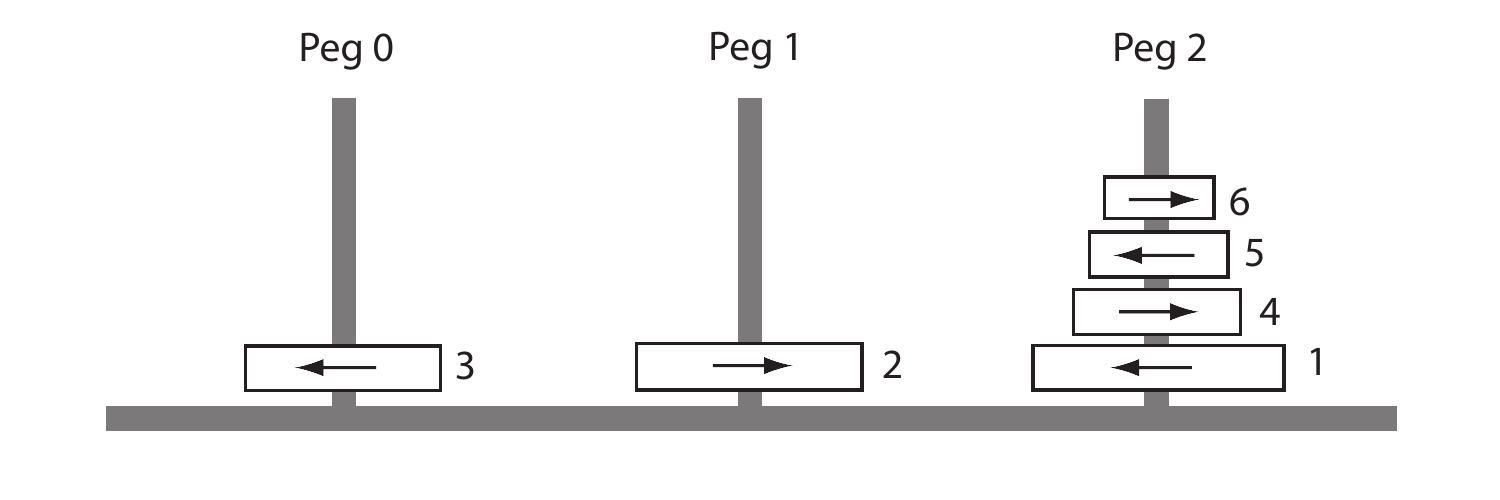}\vspace{-10pt}

\caption{\label{placementA} {\sl Assembly stage $t=.t_1t_2t_3t_4t_5t_6=.101000$ with disk positions $(P_1, P_2, P_3, P_4, P_5, P_6)=(2, 1, 0, 2, 2, 2)$. Arrows indicate the natural running direction for each disk. Moving Disk~6 (leading disk) in its natural direction will advance the solution. Moving Disk 3 against its natural direction will undo the previous move. }}
\end{figure}

\begin{remark} {\sl The allowable $2^{n+1}$ placements of the $n+1$ disks among all  possible $3^{n+1}$ placements are readily recognized by the above formula. Indeed, given the peg position $P_i$ of each Disk $i$, we can solve for $t_1\equiv -P_1$ (mod 3) and\linebreak $t_i\equiv t_{i-1}+(-1)^i(P_i-P_{i-1})$ (mod 3). Then the constellation is allowable (i.e. within the shortest solution) if and only if $t_i\in \{0,1\}$ for all $i$, in which case it corresponds to assembly stage $t=.t_1t_2\cdots t_{n+1}$.}
\end{remark}

\subsection{Our variation of the game} \label{variation} Assuming that we restrict the placement of the disks on the board to within the above solution, we make the game itself cyclic by adding a ``reset'' move, in which the final tower on Peg~2 can be lifted by the largest disk and transported back to Peg~0. The player may backtrack within the solution and may turn over any disk while holding it (but only the bottom disk during a reset move). In order to allow an observer to distinguish the two sides of each disk, we color them white~($A$) and black~($B$).

\subsection{State graphs} \label{States} Comparing the formulas of Section~\ref{shortest} with those of  Sections~\ref{sp} and \ref{edgelabels}, we see that  the graph $X_n$ can be reinterpreted as the directed state graph of the game described in Section \ref{variation}: the vertices correspond to moments when the player holds one disk in hand and the edges correspond to times when all disks are on the board. Specifically, for any point $(t,w)$ in the interior of an edge of $X_n$, say $t=.t_1t_2\cdots t_{n+1}\cdots$, the $n+1$ disks are in assembly stage $.t_1t_2\cdots t_{n+1}$. At a vertex $(t,w)$ of $X_n$, the bit length $d(t)$ of $t$, as defined in (\ref{sp}.\ref{vertices}), indicates which disk is in the player's hand. The letters of the word $w$ for any point $(t,w)$ of $X_n$ indicate the current upward-facing  color of each disk, when numbered from left (largest) to right (smallest), where the symbol ``$\,\Box$'' indicates that the corresponding disk is currently being handled. (Note that at the base vertex $x_n$, all disks are off the board, while only the largest disk is being handled.)
A given edge $e$ of $X_n$, with endpoints $(t,w)$ and $(t+1/2^{n+1},v)$, is labeled ``$x$'' if the colors of the two movable disks on the board corresponding to this edge,
 namely Disk $d(t)$ and Disk $d(t+1/2^{n+1})$, agree; otherwise $e$ is labeled ``$y$''. The edges of $X_n$ are oriented in the direction of the shortest solution to the classical puzzle. (See Figure~\ref{placementB}.)

 An edge-path in $X_n$, based at $x_n$, can be regarded as a record of the movements of the $n+1$ disks during a particular evolution of the game described in Section~\ref{variation}.\linebreak With each letter, we place a disk on the board and pick up another. (See Figure~\ref{game}.)}

 \begin{figure}[h!]

 \includegraphics[scale=.6]{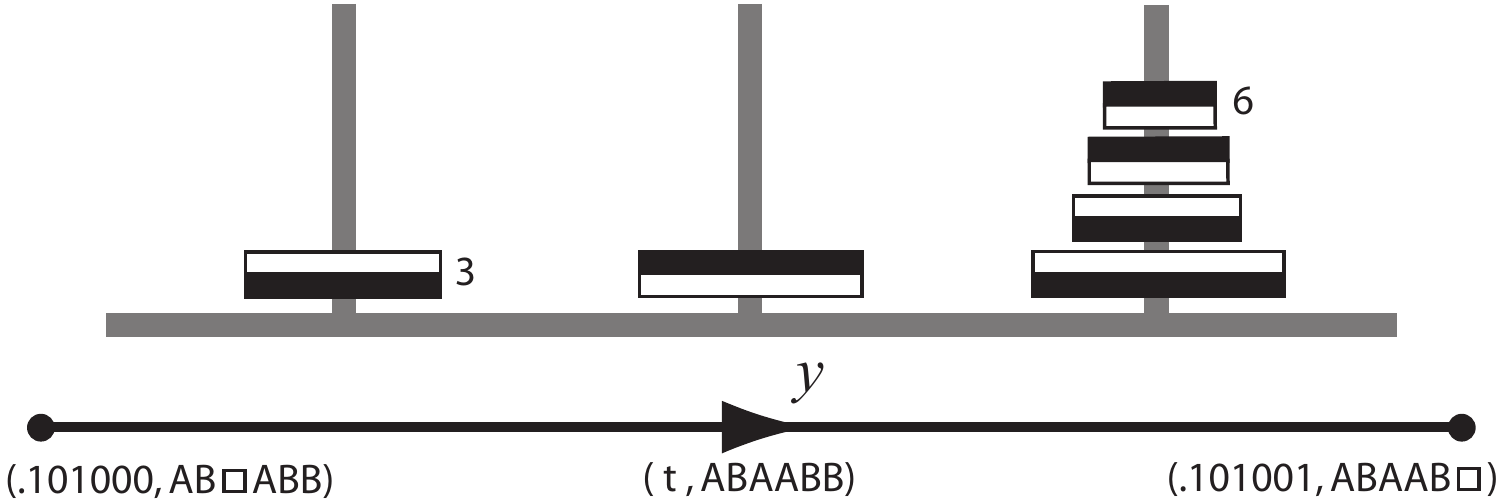}

\caption{\label{placementB} {\sl The assembly of Figure~\ref{placementA} with upward-facing colors $ABAABB$ and corresponding edge in $X_5$. (See \S\ref{States}.) At the initial vertex, Disk~3 is in the hand of the player, whereas at the terminal vertex it is Disk~6 (leading disk). The edge is labeled ``$y$'', because the colors of Disk~3 (white) and Disk~6 (black) disagree.}}
\end{figure}

\begin{figure}[h!]
\vspace{-.4in}

 \includegraphics{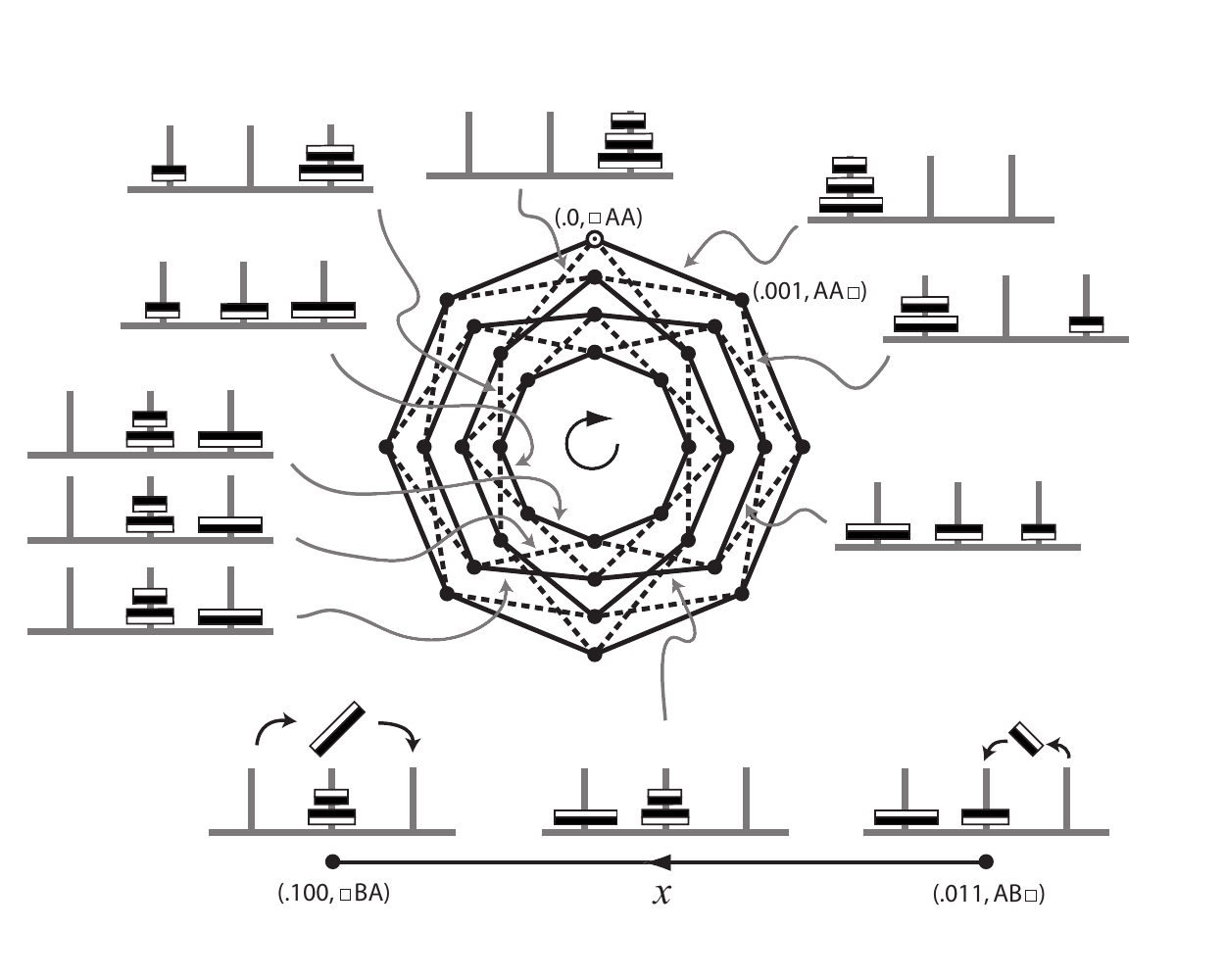}
 \vspace{-.4in}

\caption{\label{game} {\sl $X_2$ as the directed state graph of a game with 3 disks. Labels: $x=$ solid edge, $y=$ dotted edge, $A$ = white, $B$ = black.
Depicted play (edge-path): $x^{+1}y^{+1}x^{+1}x^{+1}x^{+1}y^{-1}x^{+1}x^{+1}y^{+1}y^{+1}$.
}}
\end{figure}

\pagebreak

\section{Combinatorial bonding functions $\phi_n:\Omega_{n+1}\rightarrow \Omega_n$.}\label{map}

\noindent There are natural projections $\phi_n:\Omega_{n+1}\rightarrow \Omega_n$, inducing commutative diagrams

\[\xymatrix{  \Omega_{n+1} \ar[rr]^{\text{reduce}} \ar[d]_{\phi_n} & & \Omega_{n+1}'  \ar@{<->}[r] \ar[d]_{\phi_n'}& \pi_1(X_{n+1},x_{n+1}) \ar[d]_{f_{n\#}}\\
\Omega_n \ar[rr]^{\text{reduce}} & & \Omega_n'  \ar@{<->}[r] & \pi_1(X_n,x_n)}
\]
 with $\phi_n': \Omega_{n+1}'\rightarrow \Omega_n'$ given by $\phi_n'(\omega'_{n+1}):=\phi_n(\omega'_{n+1})' =\phi_n(\omega_{n+1})'$. Using (\ref{sp}.\ref{bonding}), (\ref{sp}.\ref{linearity}) and Section~\ref{cancel}, we can describe the functions $\phi_n:\Omega_{n+1}\rightarrow \Omega_n$ in two ways, via an explicit formula or using a game description:

\subsection{An explicit formula}\label{formula}
Let $\omega_{n+1}\in \Omega_{n+1}$. Then $\omega_{n+1}=a_1a_2\cdots a_{p}$ is a finite word over the alphabet
$\{x^{+1},x^{-1},y^{+1},y^{-1}\}$, representing an edge-path which alternates between vertices of $X_{n+1}$ that project under $f_n$ to vertices of $X_n$ and vertices of $X_{n+1}$ that project to barycentric subdivision points of $X_n$. The word $\phi_n(\omega_{n+1})\in \Omega_n$ represents the edge-path of $X_n$ that traverses the vertices of $X_n$ which are visited by the projection.
In order to determine the letters of the word $\phi_n(\omega_{n+1})=b_1b_2b_3\cdots $, we need to
know when and how the edge-path $\omega_{n+1}$ transitions from one preimage $f_n^{-1}(Star(v,X^\ast_n))$ for a vertex $v$ of $X_n$ to another. (See Figures~\ref{preimage} and \ref{transitions}.)

\begin{figure}[h!]
\includegraphics{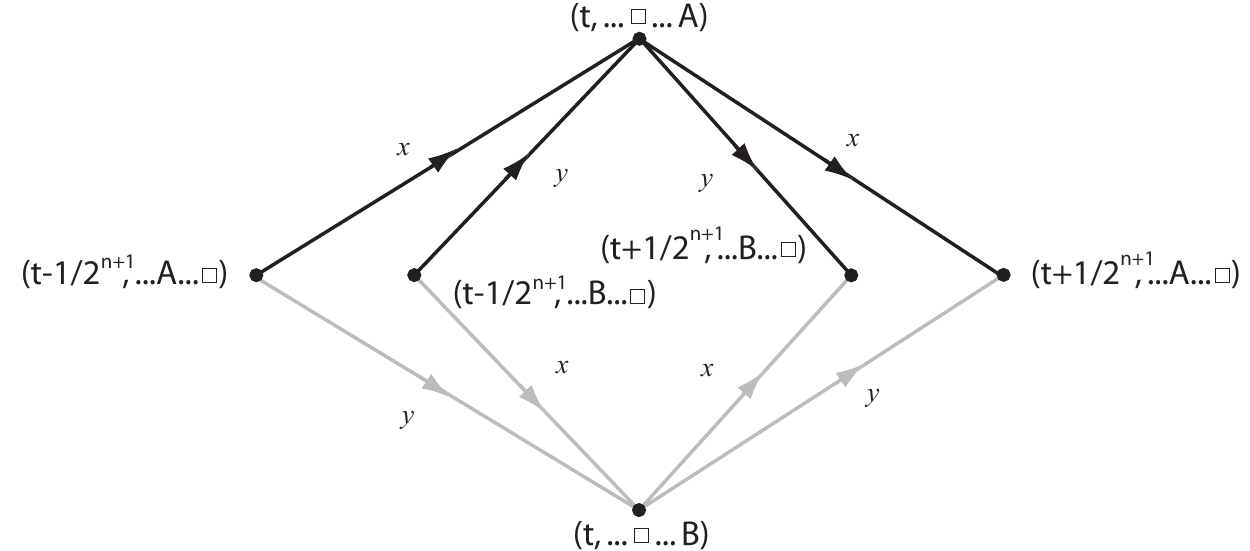}
\caption{\label{preimage} The transition rules inside the preimage of $Star(v,X^\ast_n)$ under $f_n:X_{n+1}\rightarrow X^\ast_n$ for a vertex $v$ of $X_n$.}
\end{figure}

\begin{figure}[h!]
\includegraphics{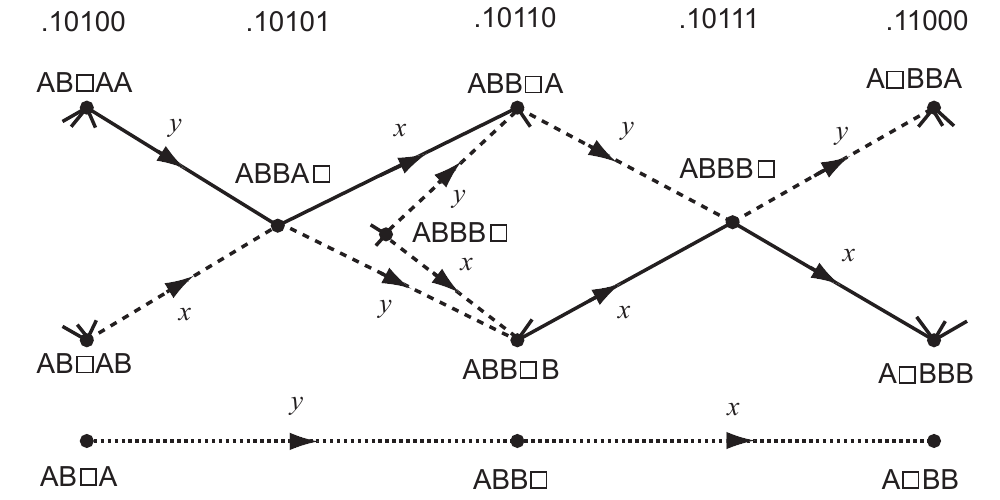}
\caption{\label{transitions} An edge-path (dotted) in $X_4$ and its projection  $\phi_3(\cdots x^{+1}\mid y^{+1} (x^{-1}y^{+1})y^{+1} \mid y^{+1} \cdots)=\cdots y^{+1}x^{+1}\cdots$ in $X_3$. The symbol ``$\mid$'' marks the positions where the edge-path  enters and leaves $f_3^{-1}(Star((.1011,ABB\Box),X_3^\ast))$.
}
\end{figure}

To this end, we pair up the letters in the word $\omega_{n+1}=(a_1a_2)(a_3a_4)(a_5a_6)\cdots$. So long as the signs in the
exponents of the letter pairs are opposite, we are still within the preimage of the same $Star(v,X^\ast_n)$, swapping the symbol ``$\Box$'' between the same two positions of the second coordinate, namely position $n+2$ (which is to be forgotten) and some fixed position $d\in\{1,2,\cdots, n+1\}$. A letter pair of matching signs signifies a transition from one vertex of $X_n$ to another, under the function $f_n$.
The corresponding edge label can be computed from $\omega_{n+1}$ using a combination of two
homomorphisms, $\chi:F_2\rightarrow \mathbb{Z}$ and $\psi: F_2\rightarrow \mathbb{Z}_2$, determined by
 $\chi(x)=\chi(y)=1$, $\psi(x)=0$ and $\psi(y)=1$.
 While $\chi$ keeps track of the exponents, the homomorphism $\psi$, when applied to the subwords of $\omega_{n+1}$ which correspond to edge-paths that stay within $f_n^{-1}(Star(v,X^\ast_n))$, keeps track of how often the letter in position $d$ changes between $A$ and $B$.
  Specifically:

\begin{definition}[$\phi_n:\Omega_{n+1}\rightarrow \Omega_n$]\label{phi} Given a word $\omega_{n+1}=a_1a_2\cdots a_{p}\in \Omega_{n+1}$, cut $\omega_{n+1}$ into subwords $\omega_{n+1}^1, \omega_{n+1}^2, \cdots, \omega_{n+1}^k $ by cutting between the two letters\linebreak  of every pair $(a_{2j-1}a_{2j})$ for which $\chi(a_{2j-1}a_{2j})\not=0$
  and calculate the word \linebreak $\phi_n(\omega_{n+1})=b_1b_2\cdots b_{k-1}\in \Omega_n$ using the following formulas (cf. Remarks~\ref{mod1} and \ref{color},\linebreak and Example~\ref{ExProj}):
  \begin{eqnarray*}
  \varepsilon_i    &=& \mbox{exponent of the last letter of $\omega_{n+1}^i$} \\
d_i           &=& \mbox{bit length of $(\varepsilon_1+\varepsilon_2+\cdots + \varepsilon_{i-1})/2^{n+1}$} \mbox{ and } d_1=1\\
color_i(d_s)& =&  \sum_{j=1}^i\delta(d_s,d_j)\psi(\omega_{n+1}^j)\\
b_i           &=& \left\{ \begin{array}{ccc}
                    x^{\varepsilon_i} & \mbox{ if } & color_i(d_i)= color_i(d_{i+1})\\
                    y^{\varepsilon_i} & \mbox{ if } & color_i(d_i)\neq color_i(d_{i+1})
                           \end{array}
                    \right.
  \end{eqnarray*}
  Here, $\delta$ is the Kronecker delta: $\displaystyle \delta(u,v)=\left\{\begin{array}{ccc}
                    1 & \mbox{ if } & u= v\\
                    0 & \mbox{ if } & u\neq v
                  \end{array}\right.$
  \end{definition}

\begin{remark}\label{mod1}
The bit length in Definition~\ref{phi} is calculated as in (\ref{sp}.\ref{vertices}) for values in $[0,1)$, upon reducing $(\varepsilon_1+\varepsilon_2+\cdots + \varepsilon_{i-1})/2^{n+1}$ modulo $1$.
\end{remark}

\begin{remark}\label{color}
{\sl For $i\in\{1,2,\dots,k-1\}$ and $d\in\{1,2,\cdots, n+1\}$,  $color_i(d)$ indicates the color of Disk $d$ at the end of $\omega_{n+1}^1 \omega_{n+1}^2 \cdots \omega_{n+1}^i\in \Omega_{n+1}$, with $0=A$ and $1=B$.}

\end{remark}

\begin{example} \label{ExProj}
 The word  $\omega_6=x^{+1}y^{-1}y^{-1}x^{+1}x^{+1}x^{-1}y^{+1}x^{+1}y^{+1}x^{-1}y^{-1}y^{+1}x^{-1}y^{+1}$ $y^{+1} y^{+1}x^{-1}y^{+1}x^{-1} y^{-1}x^{+1}y^{-1}y^{+1}x^{+1}x^{-1}y^{-1}y^{-1} y^{-1} x^{-1} y^{+1}$ of $\Omega_6$
 maps to the word \linebreak $\phi_5(\omega_{6})=y^{+1}y^{+1}y^{-1}x^{+1}x^{-1}y^{-1}$ of $\Omega_5$. The relevant values for $\omega_6^i$, $d_i$,  $\psi(\omega_6^i)$, $color_i(d_i)$, and $\varepsilon_i$ are listed below, where the boxed values signify the calculation of $b_5$. (Note: $color_i(d_s)=0$ if $d_s\not=d_j$ for all $j\in\{1,2,\dots,i\}$; otherwise, we have $color_i(d_s)=color_m(d_m)$ with $m=\max\{j\in\{1,2,\dots,i\}\mid d_s=d_j\}$. Moreover, the value of $\psi(\omega_{n+1}^k)=\psi(\omega_6^7)$ is not relevant for Definition~\ref{phi}. {\sl As for the game, $\omega_6$ ends with the player holding Disk 1. Therefore, column  ``Color'' is left blank for $i=7$. The correct interpretation of the value $color_7(1)$ is given in Section~\ref{loop} below.})
 \begin{center}
\begin{tabular}{r|l|c|c|c|c||c}
    &             & {\sl Disk}   & {\sl Flip}                &  {\sl Color} &  {\sl Advance}  &  Edge \\
$i$ &subword   $\omega_6^i$ &  $d_i$ &  $\psi(\omega_6^i)$ &\hspace{-5pt} $color_i(d_i)$\hspace{-5pt}  &   $\varepsilon_i$ &   $b_i$\\ \hline
1 &  $(x^{+1}y^{-1})(y^{-1}x^{+1})(x^{+1}x^{-1})y^{+1}$        &1& 1& $B$ &$+1$&  $y^{+1}$\\
2 &  $x^{+1}(y^{+1}x^{-1})(y^{-1}y^{+1})(x^{-1}y^{+1})y^{+1}$\hspace{-5pt}  &6 & 1& $B$ &$+1$&  $y^{+1}$  \\
3 &  $y^{+1}(x^{-1}y^{+1})x^{-1}$                              & 5 & 0& $A$&$-1$&               $y^{-1}$ \\
4 &  $y^{-1}(x^{+1}y^{-1})y^{+1}$                              &6 & 1&\fbox{$A$} &$+1$&   $x^{+1}$ \\
5 &  $x^{+1}x^{-1}$                                            &\fbox{5} &  0& \fbox{$A$}&\fbox{$-1$}& \fbox{$x^{-1}$} \\
6 &  $y^{-1}y^{-1}$                                            &\fbox{6} &  0&$A$ &$-1$&  $y^{-1}$ \\
7 &  $y^{-1}(x^{-1} y^{+1})$                                                  & 1 & {\sl hold}& &
 \end{tabular}
\end{center}\vspace{10pt}

\end{example}

\begin{remark}
The homomorphism $\phi_n':\Omega_{n+1}'\rightarrow \Omega_n'$ does not extend to a homomorphism $h:F_2\rightarrow F_2$. (Otherwise $h(x^2)=h(y^2)$, because $\phi_n'(x^{+1}y^{-1}y^{-1}x^{+1})$ equals the empty word. This would imply that $h(x^{2}y^{-1}x^{-2}y^{-1}x^{2})=h(y^{2}y^{-1}y^{-2}y^{-1}y^{2})$ is trivial. However, $\phi_n'(x^{+1}x^{+1}y^{-1}x^{-1}x^{-1}y^{-1}x^{+1}x^{+1})=x^{+1}y^{-1}y^{-1}x^{+1}$ does not equal the empty word. See also Example~\ref{HE1} below.)
\end{remark}

 \subsection{A game description}\label{gamephi} {\sl The function $\phi_n:\Omega_{n+1}\rightarrow \Omega_n$ is more easily described in terms
 of the game discussed in Section~\ref{TH}. It simply models an observer of a game, played like $\omega_{n+1}\in \Omega_{n+1}$, who ignores the smallest disk, Disk $n+2$, recording only the movements
of the remaining $n+1$ disks in form of the word  $\phi_n(\omega_{n+1})\in \Omega_n$.}

\begin{remark}\label{backandforth}

{\sl It might be worthwhile examining exactly how this observer processes the disk information during those stretches $\omega^i_{n+1}$ of $\omega_{n+1}\in \Omega_{n+1}$ when the configurations remain within the preimage of one fixed $Star(v,X^\ast_n)$ of a vertex $v$ of $X_n$, as depicted in Figure~\ref{preimage}. Put $d=d(t)\in \{1,2,\dots, n+1\}$. Then the left-hand side of Figure~\ref{preimage} has Disk $d$ on one peg and the right-hand side on another. The smallest disk, Disk~$n+2$, consistently occupies the remaining peg (or, if $d=1$, is atop the entire disk stack), white-side-up for the upper half of Figure~\ref{preimage} and black-side-up for the lower half. The observer, who ignores Disk~$n+2$, will not (indeed cannot) document any off-and-on or back-and-forth by Disk~$d$ between its two pegs, while Disk~$n+2$ intermittently moves off and on its otherwise fixed position. Only after Disk~$d$ settles down, will the observer add a letter to the word in progress, $\phi_n(\omega_{n+1})$, unless the game has ended. This does not happen until Disk $n+2$ moves to a new peg and some other disk, say Disk $d'$ with $d\neq d'\in \{1,2,\cdots,n+1\}$, is picked up. At that moment, the added letter is determined by comparing the colors of $d$ and $d'$, and by whether $d$ or $d'$ is the leading disk among the $n+1$ largest disks. (See Figure~\ref{undecided}, with $n+2=4$ and $d=3$.)}
\end{remark}

\vspace{-10pt}

\begin{figure}[h!]
\includegraphics{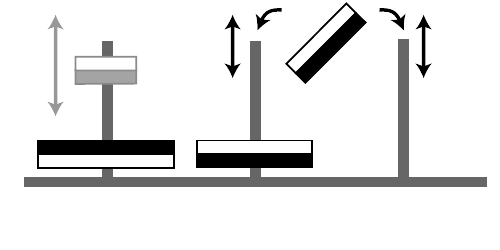}\vspace{-.2in}
\caption{\label{undecided} {\sl Disk movements corresponding to edge-paths within a subgraph of $X_3$ as depicted in Figure~\ref{preimage}, where the indicated configuration corresponds to the top vertex $(.0110,BA\Box A)$. Disk~3  moves freely among Peg~1 and Peg~2 (possibly turning), while the smallest disk, Disk~4 (depicted in gray), intermittently moves off and on Peg~0 (possibly turning).}}
\end{figure}

\subsection{Edge-loops}\label{loop}
A finite word $\omega_1$ over the alphabet $\{x^{+1},x^{-1},y^{+1},y^{-1}\}$ is an element of $\Omega_1$ if and only if
 $\chi(\omega_1)\equiv 0$ (mod 4) and $\psi(\omega_1)=0$.
A finite word $\omega_{n+1}$ over the alphabet $\{x^{+1},x^{-1},y^{+1},y^{-1}\}$ is an element of $\Omega_{n+1}$ if
 and only if $\chi(\omega_{n+1})\equiv 0$ (mod $2^{n+2}$) and
$color_k(d)=0$ for all $d\in \{1, 2, \ldots, n+1\}$, where $k$ is as in Definition~\ref{phi}.
({\sl Note that, if $\chi(\omega_{n+1})\equiv 0$ (mod $2^{n+2}$), $color_k(1)$ actually indicates the color of Disk $n+2$ at the very end of $\omega_{n+1}$. To see why, simply append the two-letter word $x^{+1}x^{+1}$ to $\omega_{n+1}$ and form the word $\xi_{n+1}=\omega_{n+1}x^{+1}x^{+1}$. Then, by Remark~\ref{color},  $color_k(1)$ indicates the color of Disk 1 at the end of $\xi^1_{n+1}\xi^2_{n+1}\cdots \xi^k_{n+1}=\omega_{n+1}x^{+1}$, which is the same as the color of Disk $n+2$ at the end of $\omega_{n+1}$. Therefore, the stated formula actually checks whether the
$n+1$} {\em smallest} {\sl disks are returned to ``white'',
rather than the $n+1$} {\em largest} {\sl disks.})

Coincidentally, the formula for $\Omega_{n+1}$ with $n=0$ agrees with that for $\Omega_1$.

\section{The fundamental group of $\mathbb{M}$}\label{PiM}
As indicated in Section~\ref{cech}, we may now begin applying the results from \cite{FZ} to obtain an explicit and systematic description of  $\pi_1(\mathbb{M})$ in terms of the combinatorial bonding functions $\phi_n:\Omega_{n+1}\rightarrow \Omega_n$ from Section~\ref{formula}/\ref{gamephi} and thereby in terms of the Towers of Hanoi puzzle from Section~\ref{TH}.

\subsection{Locally eventually constant sequences}\label{LEC}  As in \cite{FZ}, we consider
 \begin{eqnarray*}
 \Omega &=&\lim_{\longleftarrow}\big(\Omega_1 \stackrel{\phi_1}{\longleftarrow} \Omega_2 \stackrel{\phi_2}{\longleftarrow} \Omega_3
\stackrel{\phi_3}{\longleftarrow} \cdots\big) \mbox{ and}\\
G &= &\lim_{\longleftarrow}\big(\Omega_1' \stackrel{\phi_1'}{\longleftarrow} \Omega_2' \stackrel{\phi_2'}{\longleftarrow} \Omega_3'
\stackrel{\phi_3'}{\longleftarrow} \cdots\big)\cong\check{\pi}_1(\mathbb{M}).
\end{eqnarray*}
We call a sequence $(g_n)_n\in G$
 {\em locally eventually  constant} if for every $n$, the sequence  $(\phi_n\circ \phi_{n+1}\circ\cdots\circ \phi_{k-1}(g_k))_{k>n}$ of unreduced words is eventually constant. We put
\[\mathcal G=\{(g_n)_n\in G\mid (g_n)_n \mbox{ is locally eventually constant}\}.\]
For a sequence $(g_n)_n\in {\mathcal G}$ we define the {\em stabilization} $\overleftarrow{(g_n)_n}=(\omega_n)_n\in \Omega$ by \linebreak $\omega_n=\phi_n\circ \phi_{n+1}\circ\cdots\circ \phi_{k-1}(g_k)$, for sufficiently large $k$, and
$\overleftarrow{\mathcal G}=\{\overleftarrow{(g_n)_n}\mid (g_n)_n\in {\mathcal G}\}$.

The {\em reduction} of a sequence $(\omega_n)_n\in \Omega$ is defined by $(\omega_n)'_n=(\omega'_n)_n\in G$. Each $(\omega_n)_n\in \overleftarrow{\mathcal G}$ corresponds to a unique $(g_n)_n\in \mathcal G$ by way of $(g_n)_n=(\omega_n)'_n$ and $(\omega_n)_n=\overleftarrow{(g_n)_n}$ (cf. \cite[Remark~4.7]{FZ}): \[\Omega \supseteq \overleftarrow{\mathcal G} \longleftrightarrow {\mathcal G}  \subseteq G \]

\begin{remark} See Example~\ref{HE1} for a sequence
  which is {\em not} locally eventually constant.
\end{remark}

The combinatorial structure of $\pi_1(\mathbb{M})$ is  captured by the following theorem:

\begin{theorem}[Theorem A of \cite{FZ}]\label{structure}
 The word sequences of $\overleftarrow{\mathcal G}$ form a group under the binary operation $\ast$, given by termwise concatenation, followed by reduction and restabilization: $(\omega_n)_n\ast(\xi_n)_n=\overleftarrow{(\omega_n\xi_n)'_n}$. Moreover, $\overleftarrow{{\mathcal G}}\cong \pi_1(\mathbb{M})$.
\end{theorem}

\subsection{A systematic description of the elements of the fundamental group}\label{algorithm}
Theorem~\ref{structure} describes the elements of the group $\pi_1(\mathbb{M})$ in terms of stabilizing coherent sequences of an inverse system. In order to facilitate a better combinatorial understanding of this result, we now present an algorithm which systematically describes the elements of
$\overleftarrow{\mathcal G}\cong \pi_1(\mathbb{M})$  via a finite set of rules, along with a corresponding game interpretation in Section~\ref{visualization} below. By randomly varying the parameters of this algorithm, every element of the fundamental group can be produced. The randomized algorithm recursively creates a generic sequence $g_1, \omega_1, g_2, \omega_2, g_3, \dots$ with $(\omega_n)_n\in \Omega$ and $(g_n)_n=(\omega_n)'_n$. In order to ensure that $(\omega_n)_n=\overleftarrow{(g_n)_n}$, it also arranges that for every $n\in \mathbb{N}$ there is a $K(n)\in \mathbb{N}$ such that $\omega_n=\phi_n\circ \phi_{n+1}\circ\cdots\circ \phi_{K(n)-1}(\omega'_{K(n)})$. (See Remark~\ref{stab} below.)\vspace{5pt}

{\bf Step 1 (Initialization):} The term $g_1$ can be any element of $\Omega_1'\cong \pi_1(X_1)$.  Considering the free product decomposition
 $\pi_1(X_1)=\pi_1(X_0)\ast \pi_1(X_0)\ast \pi_1(X_0)$ from Section~\ref{cech}, we fix the following set of free generators for $\Omega_1'$:
\vspace{5pt}

\begin{center}  \parbox{2.5in}{

$A_1=\mbox{\hspace{7pt} } x^{-1}A$ \hspace{15pt}  where $A=y^{+1}y^{+1}x^{-1}$

$B_1=\mbox{\hspace{7pt} } y^{-1}B$  \hspace{15pt} where $B=x^{+1}y^{+1}x^{-1}$

$C_1=\mbox{\hspace{7pt} } y^{+1}C$ \hspace{15pt} where $C=x^{-1}y^{+1}x^{-1}$

$A_2=Xy^{+1}A$ \hspace{17pt} where $X=x^{+1}y^{+1}$

$B_2=Xx^{+1}B$

$C_2=Xx^{-1}C$

$A_3=Yx^{+1}A$ \hspace{17pt} where $Y=x^{+1}x^{+1}$

$B_3=Yy^{+1}B$

$C_3=Yy^{-1}C$ }
\parbox{2in}{\includegraphics{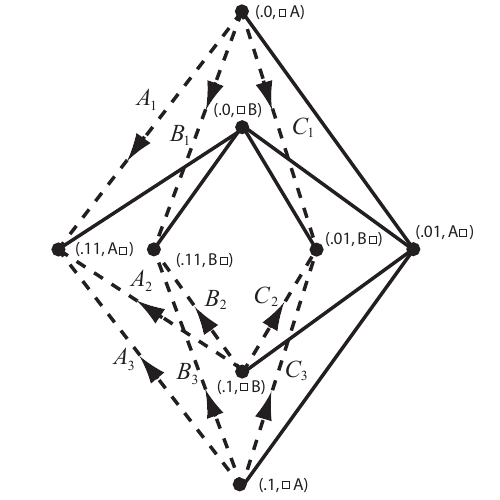}}\end{center}
\vspace{5pt}

Here, the words $A, B, C, X$ and $Y$ connect the nine generating edges through the maximal subtree (drawn with solid lines) to the base point.

In order to get a generic $g_1$, we randomly write down an element in the free group on the above nine generators $A_1, B_1, C_1, A_2, B_2, C_2, A_3, B_3, C_3$ and their inverses, and reduce the resulting word in $F_2$, the free group on $\{x,y\}$. Say, $g_1=c_1c_2\cdots c_r$ with $c_i\in \{x^{+1}, x^{-1}, y^{+1}, y^{-1}\}$. We then create an enveloping word $\omega_1\in \Omega_1$ with $\omega_1'=g_1$, whose role it is to force stabilization at level 1. To this end, we perform a finite number of reverse cancellations (unrelated to any previous cancellations) by randomly choosing an insertion sequence for $g_1$ in the form $(p_1, z_1, p_2, z_2, \ldots , p_s,z_s)$ with $1\leqslant p_i \leqslant r+1 +2(i-1)$ and $z_i\in \{x^{+1}, x^{-1}, y^{+1}, y^{-1}\}$,  meaning that at position $p_1$ of the word $g_1$ we insert $z_1z_1^{-1}$ (we prepend it if $p_1=1$; we append it if $p_1=r+1$), at position $p_2$ of the resulting word we insert $z_2z_2^{-1}$, etc. (Here, we use the usual law of exponents: e.g., $(x^{-1})^{-1}=x$.) The end result is $\omega_1$.\linebreak  Finally, we randomly choose $K(1)\in \mathbb{N}$ with $K(1)>1$ in order to mark the index by which we intend to force stabilization at level 1.\vspace{5pt}

{\bf Step 2 (Recursion):} Suppose  $\omega_n=b_1b_2\cdots b_{k-1}\in \Omega_n$ and positive integers $K(1), K(2), \dots, K(n)$ with $K(m)>m$ for all $1\leqslant m \leqslant n$ are given. In order to obtain a generic $g_{n+1}$ from $\omega_n$, we return to the decomposition $\pi_1(X_{n+1})=\pi_1(X_n)\ast F_3\ast F_3 \ast \cdots \ast F_3$ from Section~\ref{cech}, with one free factor of $F_3$ for each vertex of $X_n$. First, calculate the word $u_1v_1u_2v_2\cdots u_{k-1}v_{k-1}\in \Omega_{n+1}$ representing the edge-loop of $X_{n+1}$ corresponding to the $g_n$-lift, as defined in (\ref{sp}.\ref{lifts}), of the projection into $X_n^\ast$ of the edge-loop $\omega_n$ of $X_n$ starting at $x_n$. (To do this, note that $\chi(u_i)=\chi(v_i)=\chi(b_i)$; substitute $\omega_{n+1}^1=u_1$, $\omega_{n+1}^i=v_{i-1}u_i$, $\omega_{n+1}^k=v_{k-1}$ into the formulas of Definition~\ref{phi} and recursively calculate the values for $\varepsilon_i$, $d_i$, $color_i(d_i)$ and $\psi(\omega_{n+1}^i)$, based on $b_1b_2\cdots b_{k-1}$, noting that $\delta(d_i,d_{i+1})=0$; also note that $\psi(v_i)=color_i(d_{i+1})$.) Next, randomly choose elements $W_1, W_2,\ldots,W_k$ in the free group $F_3$ on $\{A_1, B_1, C_1\}$, regarded as a subgroup of $F_2$ with the above substitutions $A_1=(x^{-1}y^{+1})(y^{+1}x^{-1}), B_1=(y^{-1}x^{+1})(y^{+1}x^{-1})$, and $C_1=(y^{+1}x^{-1})(y^{+1}x^{-1})$. (We choose each $W_i$ reduced as a word over the alphabet $\{A_1^{\pm 1}, B_1^{\pm 1}, C_1^{\pm 1}\}$. After the indicated substitutions, it need not be reduced over the alphabet $\{x^{\pm 1},y^{\pm 1}\}$.) Now define
\[
\begin{array}{ccl}
\omega_{n+1}^1 & :=& W_1u_1\\

\omega_{n+1}^2 &:=& v_1 W_2 u_2 \\

\omega_{n+1}^3 &:=& v_2 W_3 u_3 \\

& \vdots & \\

 \omega_{n+1}^{k-1} &:=& v_{k-2} W_{k-1} u_{k-1} \\

\omega_{n+1}^k &:=& v_{k-1} W_k

\end{array}\]
 Observe that the letters of each of the concatenated words $\omega_{n+1}^1u_1^{-1}$, $v_i^{-1}\omega_{n+1}^{i+1}u_{i+1}^{-1}$, and $v_{k-1}^{-1}\omega_{n+1}^k$ come in consecutive pairs of opposite exponents.  Next, form words $\widehat{\omega}_{n+1}^i$ by subjecting each reduced word $(\omega_{n+1}^i)'$ to a random insertion sequence of the form  $(p_{i,1}, z_{i,1}, p_{i,2}, z_{i,2}, \ldots , p_{i,s_i},z_{i,s_i})$ as above, such that $s_i>0$ if $(\omega_{n+1}^i)'$ is the empty word, and such that the letters of each of the concatenated words $\widehat{\omega}_{n+1}^1u_1^{-1}$,\linebreak $v_i^{-1}\widehat{\omega}_{n+1}^{i+1}u_{i+1}^{-1}$, and $v_{k-1}^{-1}\widehat{\omega}_{n+1}^k$ still come in consecutive pairs of opposite exponents. Put \[\omega_{n+1}:=\widehat{\omega}_{n+1}^1\widehat{\omega}_{n+1}^2 \cdots \widehat{\omega}_{n+1}^k.\]
Then $\omega_n=\phi_n(\omega_{n+1})$.

If there is an $m$ with $K(m)=n+1$ and $\omega_m\neq\phi_m\circ \phi_{m+1}\circ\cdots\circ \phi_{n}(\omega'_{n+1})$, then we have failed to stabilize the sequence as planned and need to choose different words $W_1, W_2, \cdots, W_k$. In this event, we simply restrict the choice of $W_i$ by requiring it to be nonempty if $v_{i-1}=u_i^{-1}$, which will guarantee that $\omega_n=\phi_n(\omega'_{n+1})$, so that $\omega_m=\phi_m\circ \phi_{m+1}\circ\cdots\circ \phi_{n}(\omega'_{n+1})$ for all $m$ with $K(m)=n+1$.
  Finally, put $g_{n+1}=\omega'_{n+1}$ and randomly choose $K(n+1)\in \mathbb{N}$ with $K(n+1)>n+1$.

\begin{remark}\label{stab} Let $(\omega_n)_n\in \Omega$. Consider  $\omega_n=\phi_n\circ \phi_{n+1}(\omega_{n+2})$, $\nu_n=\phi_n\circ \phi_{n+1}(\omega'_{n+2})$, and $\xi_n=\phi_n(\omega'_{n+1})=\phi_n( \phi_{n+1}(\omega_{n+2})')=\phi_n\circ \phi'_{n+1}(\omega'_{n+2})$. Then the word $\nu_n$ can be obtained from $\omega_n$ by performing some letter cancellations, and $\xi_n$ in turn  can be obtained from $\nu_n$ by performing some letter cancellations. (Recall that  $\phi_n:\Omega_{n+1}\rightarrow \Omega_n$ represents the edge-path projection induced by the continuous function $f_n:X_{n+1}\rightarrow X_n$ and that cancellations generate the edge-path homotopies in graphs.) Hence, if $\omega_n=\xi_n$, then $\omega_n=\nu_n$. Therefore, if $\omega_n=\phi_n(\omega'_{n+1})$, then  $\omega_n=\phi_n\circ\phi_{n+1}(\omega'_{n+2})$. Similarly, if $\omega_n=\phi_n\circ \phi_{n+1}\circ\cdots\circ\phi_{k-1}(\omega'_k)$ then
$\omega_n=\phi_n\circ \phi_{n+1}\circ\cdots\circ\phi_{k}(\omega'_{k+1})$.
\end{remark}

\subsection{The game interpretation}\label{visualization}
  {\sl The word $\omega_n$ instructs the player on how to move $n+1$ disks. If we interfere with these movements by adding one more disk, Disk~$n+2$ (now the smallest disk), the player will have to move it out of the way every other turn (and each time there is only one peg for it to go). If, in the process, Disk $n+2$ is always placed on the board white-side-up, then the resulting movements of all $n+2$ disks combined generate the  word $u_1v_1u_2v_2\cdots u_{k-1}v_{k-1}$ of Section~\ref{algorithm}.

 Before moving Disk $n+2$ to a new peg, the player may choose to stall the game using the tactics described in Remark~\ref{backandforth}, thus generating the words $\widehat{\omega}^i_{n+1}$.

  A letter cancellation in the word $\omega_n$ means the elimination from the record of two consecutive actions that undo each other. So, if a sequence $(\omega_n)_n\in \Omega$ has the property that for every $n$ there is a $k$ with $\omega_n=\phi_n\circ \phi_{n+1}\circ\cdots\circ \phi_{k-1}(\omega'_k)$, then it is in a state of ``maximal consistent reduction'':

  While an observer of the movements of the largest $n+1$ disks might seemingly detect some cancelling disk interactions, none of these will be stricken from the record if we reveal the movements of the largest $k+1$ (or more) disks to this observer and allow him or her to eliminate all observed cancelling disk interactions among the largest $k+1$ (or more) disks before recording the movements of the largest $n+1$ disks.

  The set $\overleftarrow{\mathcal G}$ then consists precisely of those game sequences $(\omega_n)_n\in \Omega$ which, while consistently moving an increasing number of disks, are in this state of maximal consistent reduction.}

\section{The generalized Cayley graph}\label{CayM}

In \cite{FZ} we have shown that, given any one-dimensional path-connected compact metric space, there is (in principle) an $\mathbb{R}$\nobreakdash-tree which (to the extent possible) functions like a Cayley graph for the fundamental group.
 (Recall that an $\mathbb{R}$-tree is a metric space in which every two points are connected by a unique arc, and this arc is isometric to an interval of the real line.)  We now give an explicit and systematic description of this generalized Cayley graph for $\pi_1({\mathbb M})$ in terms of a two-letter word calculus, as well as in terms of our Towers of Hanoi game from Section~\ref{TH}. Before we do, we adapt the necessary terminology from \cite{FZ} to the given set-up.

\subsection{Word sequences} We now need to consider edge-paths in $X_n$, which start at $x_n$, but end at any vertex of $X_n$. For each $n\in \mathbb{N}$, these edge-paths are described by the words of ${\mathcal W}$, as defined in Section~\ref{cancel}. However, setting up
a consistent projection/reduction scheme for these words necessitates new notation, because half of the vertices of $X_{n+1}$ are not projected to vertices of $X_n$. For this purpose, we use the notation $a_1a_2\cdots a_{p-1}/a_p$ to represent any edge-path in $X_n$, which starts at $x_n$, follows the edge-path $a_1a_2\cdots a_{p-1}$, and ends somewhere in the interior of the edge spanned by the endpoints of the two edge-paths $a_1a_2\cdots a_{p-1}$ and $a_1a_2\cdots a_{p}$, without visiting any further vertices of $X_n$. Symbolically, we define the set
\[{\mathcal W}_+= {\mathcal W}\cup \{a_1a_2\cdots a_{p-1}/a_{p}\mid  a_1a_2\cdots a_{p}\in {\mathcal W}, p\geqslant 1 \}.\]
 Reduction of words in ${\mathcal W}_+$ is carried out exactly as before, where cancellations of subwords of the form $z^{+1}z^{-1}$ across the symbol ``/'' take on the following form: \[a_1a_2\cdots a_{p}z^{+1}/z^{-1}\mapsto a_1a_2\cdots a_{p}/z^{+1}.\]
 The functions $\phi_n:\Omega_{n+1}\rightarrow \Omega_n$ and $\phi'_n:\Omega'_{n+1}\rightarrow \Omega'_n$  naturally extend to functions $\phi_n:{\mathcal W}_+\rightarrow {\mathcal W}_+$ and  $\phi'_n:{\mathcal W}'_+\rightarrow {\mathcal W}'_+$, respectively, with commuting diagrams:
\[\xymatrix{  {\mathcal W}_+ \ar[rr]^{\text{reduce}} \ar[d]_{\phi_n} & & {\mathcal W}'_+ \ar[d]_{\phi'_n} \\
{\mathcal W}_+ \ar[rr]^{\text{reduce}} & & {\mathcal W}'_+}
\]

\begin{definition}
For $\omega_{n+1}\in {\mathcal W}_+$, we define $\phi_n(\omega_{n+1})$ as follows:
 \begin{itemize}
 \item[(i)] If $\omega_{n+1}=a_1a_2\cdots a_p$ with $p$ even, then $\phi_n(\omega_{n+1})$ is defined as in Section~\ref{map}.
 \item[(ii)] If $\omega_{n+1}=a_1a_2\cdots a_p$ with $p$ odd, then $\phi_n(\omega_{n+1}):=b_1b_2\cdots b_{k-2}/b_{k-1}$, where
 $\phi_n(a_1a_2\cdots a_pa_p)=b_1b_2\cdots b_{k-1}$.
 \item[(iii)] If $\omega_{n+1}=a_1a_2\cdots a_{p-1}/a_p$ with $p$ even, then $\phi_n(\omega_{n+1}):=\phi_n(a_1a_2\cdots a_{p-1})$.
 \item[(iv)] If $\omega_{n+1}=a_1a_2\cdots a_{p-1}/a_p$ with $p$ odd,  then $\phi_n(\omega_{n+1}):=\phi_n(a_1a_2\cdots a_p)$.
\end{itemize}
\end{definition}

\noindent  As before, $\phi_n'(\omega'_{n+1}):=\phi_n(\omega'_{n+1})' =\phi_n(\omega_{n+1})'$.

\begin{remark} {\sl From the perspective of our game, what prompts an observer of $n+1$ disks to put down the last letter of $\omega_n=a_1a_2\cdots a_{p}$ is the fact that the last disk was picked up. (See also Remark~\ref{backandforth}.) If the game ends, instead, with the last disk (and hence all disks) still on the board, one would write $\omega_n=a_1a_2\cdots a_{p-1}/a_p$.\linebreak So, our interpretation of $\phi_n:{\mathcal W}_+\rightarrow {\mathcal W}_+$ as modeling an observer of $n+2$ disks who only records the movements of the largest $n+1$ disks, remains the same.}
\end{remark}

Analogously to Section~\ref{LEC}, we define the following sets:
\begin{eqnarray*}
\mathcal W\mathcal S & =& \lim_{\longleftarrow}\big({\mathcal W}_+ \stackrel{\phi_1}{\longleftarrow} {\mathcal W}_+\stackrel{\phi_2}{\longleftarrow} {\mathcal W}_+
\stackrel{\phi_3}{\longleftarrow} \cdots\big) \\
 R &=& \lim_{\longleftarrow}\big({\mathcal W}'_+ \stackrel{\phi_1'}{\longleftarrow} {\mathcal W}'_+ \stackrel{\phi_2'}{\longleftarrow} {\mathcal W}'_+
\stackrel{\phi_3'}{\longleftarrow} \cdots\big)\\
{\mathcal R}&=&\{(r_n)_n\in R\mid (r_n)_n \mbox{ is locally eventually constant}\}\\
 \overleftarrow{\mathcal R}&=&\{ \overleftarrow{(r_n)_n} \mid (r_n)_n \in {\mathcal R}\}
 \end{eqnarray*}

\subsection{Equivalence} We say that a word sequence $(\omega_n)_n\in {\mathcal WS}$ is of {\em terminating type} if there is an $N\in \mathbb{N}$ such that $\omega_n\in {\mathcal W}$ for all $n\geqslant N$. There is a natural word sequence analog of the equivalence ``$0.999\dots\doteq 1.000\dots$'' in the decimal expansions of real numbers, which we now describe.

\begin{remark}\label{terminating}
 {\sl If $(\omega_n)_n\in {\mathcal WS}$ is of terminating type,  there is a $d\in \{1,2,\dots,N+1\}$ such that for every $n\geqslant d-1$, at the end of the play $\omega_n$, Disk~$d$ is in the hands of the player and Disks $d+1,d+2, \dots, n+1$ are all stacked up on the same peg (or, if $d=1$, are possibly all off the board). Unless $(\omega_n)_n$ is the sequence of empty words, there is exactly one word sequence, which records the exact same disk movements as $(\omega_n)_n$, except that Disk $d$ is not being picked up at the end of every word, and there are always exactly four distinct word sequences, which record the exact same disk movements as $(\omega_n)_n$, except that Disk $d$ is set down at the end of every word. (Since our disk configurations are restricted to within the shortest solution, we can set down Disk $d$ on two different pegs with one consistent choice of color facing up.) We call these word sequences} {\em equivalent}.
\end{remark}

Formally, we have the following  definition \cite[Definition~3.10]{FZ}.

\begin{definition}[Equivalence]
Given $(\omega_n)_n,(\xi_n)_n\in\mathcal W\mathcal S$, we write $(\omega_n)\doteq (\xi_n)_n$ if there is an $N\in \mathbb{N}$ such that $\omega_n=a_{n,1}a_{n,2}\cdots a_{n,m_n}$ for all $n\geqslant N$, i.e., $(\omega_n)_n$ is of terminating type, and either $\xi_n=a_{n,1}a_{n,2}\cdots a_{n,m_n-1}/a_{n,m_n}$ for all $n\geqslant N$ or  $\xi_n=a_{n,1}a_{n,2}\cdots a_{n,m_n}/a_{n,m_n+1}$ for all $n\geqslant N$ and some $a_{n,m_n+1}$. This induces an equivalence relation on $\mathcal W\mathcal S$ which we denote by the same symbol, ``$\doteq$''. Given a subset $S\subseteq \mathcal W\mathcal S$, we denote by $\dot{S}$ the set resulting from replacing every element of $S$ by an equivalent one from $\mathcal W\mathcal S$ of terminating type, whenever possible.
\end{definition}

\begin{remark} An equivalence class contains at most one word sequence of terminating type. If it contains an element of terminating type, then it contains six elements total, unless it contains the sequence of empty words, in which case it contains five. All other equivalence classes are singletons.
\end{remark}

\subsection{Completion} Before we can begin measuring distances between word sequences, we need to address their limiting behavior. Specifically, consider a word sequence $(\omega_n)_n\in \mathcal{WS}$ as depicted in  Figure~\ref{CompleteFigure} (left), where there is a vertex $(t,v)$ of $X_n$ such that every edge-path $\omega_k$ ($k>n+1$) comes to within one edge of a vertex of $X_k$ that maps to $(t,v)$, although $\omega_n$ does not proceed to visit vertex $(t,v)$ at that time. In such a situation, we insert the corresponding detour to vertex $(t,v)$ into the edge-path $\omega_n$ by way of Definition~\ref{compl} below. A game interpretation is provided in the following remark.

\begin{figure}[h]

\begin{tabular}{|c|c|}\hline

\includegraphics[scale=.75]{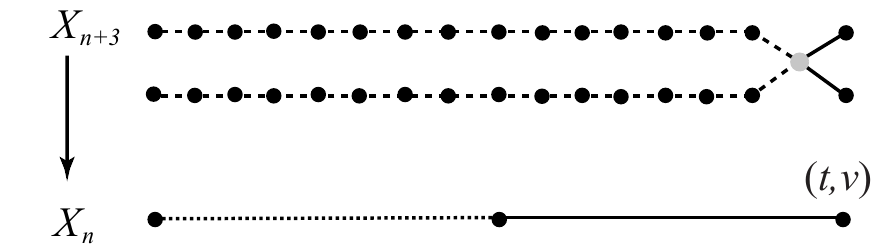} & \includegraphics[scale=.35]{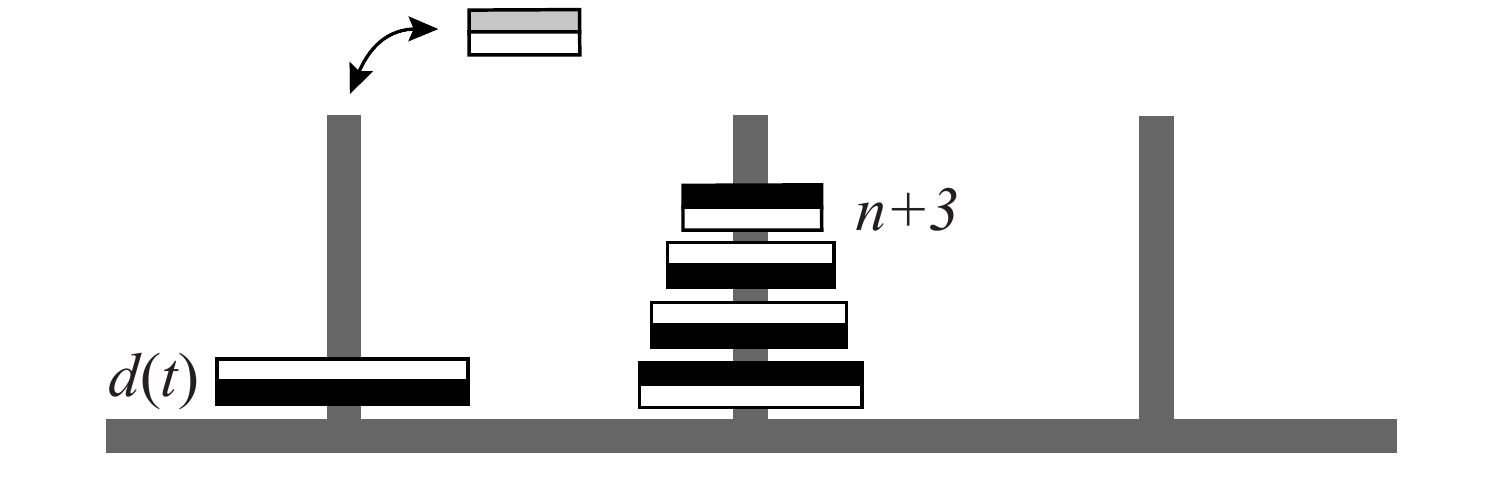} \\ \hline

\includegraphics[scale=.75]{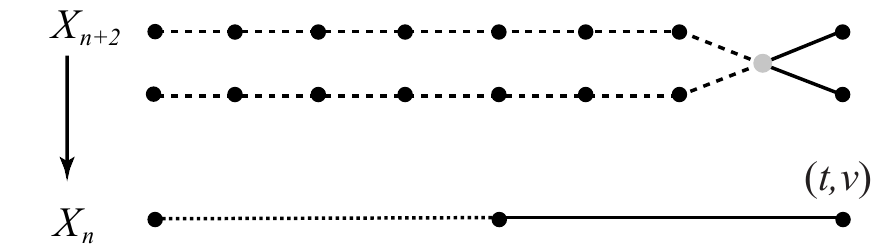} & \includegraphics[scale=.35]{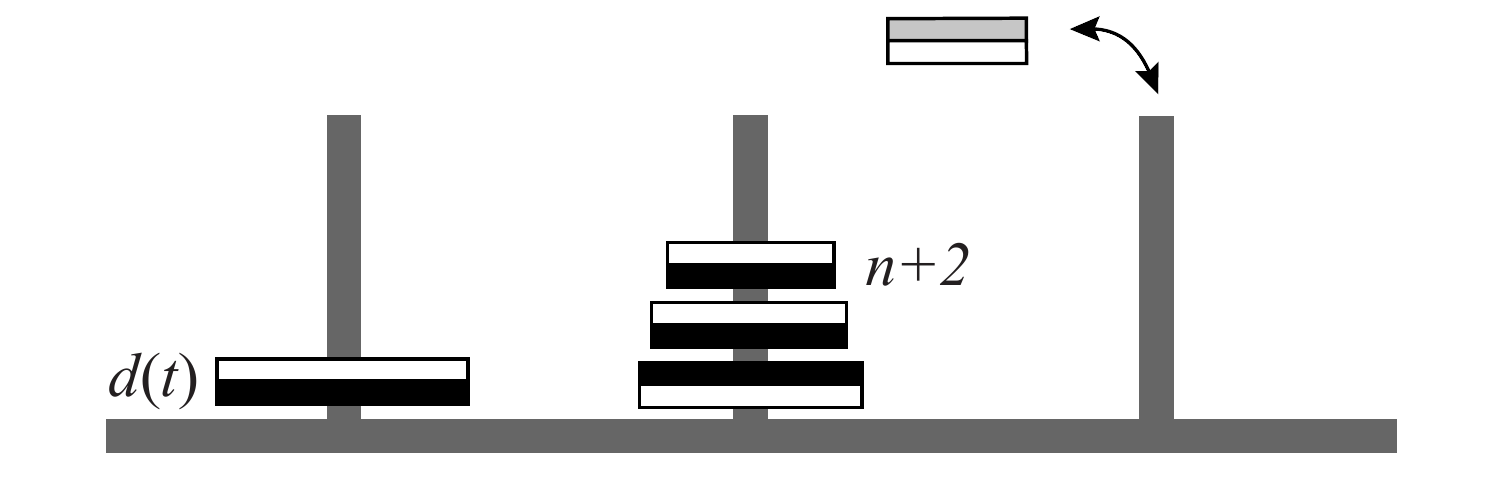} \\ \hline

\end{tabular}
\caption{\label{CompleteFigure} {\em Left:}  The dotted edge-paths $\omega_{n+2}, \omega_{n+3}, \dots$  come to within one edge of a vertex projecting to $(t,v)$, but the projection $\omega_n$ does not visit $(t,v)$. {\em Right:} {\sl The corresponding disk movements. As $\omega_{n+2}, \omega_{n+3}, \dots$ pass the gray vertex, the smallest disk (gray) gets picked up and immediately returned to the same peg (possibly turned), while Disk $d(t)$ remains on the board (exposed).}}
\end{figure}

\begin{remark}
{\sl Similar to Remark~\ref{terminating}, for a given $(\omega_n)_n\in {\mathcal WS}$ and fixed $1\leqslant d \leqslant n+1$, the player might be stacking up Disks $d+1, d+2, \dots, k$  during  corresponding moments of $\omega_k$ with $k>n+1$,
 seemingly preparing to pick up Disk $d$, but without actually doing so. (See Figure~\ref{CompleteFigure}, right, with $d(t)=d$.) Indeed, Disk $k$ might be added to this stack  during $\omega_k$ at a time when Disk $k+1$ is perceived as being handled by the player, only to be added to the stack itself during $\omega_{k+1}$ while Disk~$k+2$ is being handled, etc. From the word $\omega_n$, one cannot discern that, as more and more disks are being moved, the player is coming infinitesimally close to picking up Disk~$d$.
We remedy this situation by temporarily setting Disk $k+1$ on Disk~$k$ during $\omega_k$, while picking up Disk $d$ in the process---a move which is noted by the observer recording $\omega_n$---and checking if the trend persists. If it does, Definition~\ref{compl} below inserts the appropriate detour to vertex $(t,v)$ into the edge-path $\omega_n$.}
\end{remark}

The following definition is adapted from \cite[Definition~4.9]{FZ}:

\begin{definition}[Completion]  \label{compl} The completion $\overline{(\omega_n)_n}=(\tau_n)_n$ of a word sequence $(\omega_n)_n\in \mathcal W\mathcal S$ can be calculated as follows.
Fix $n\in \mathbb{N}$. For each $k>n+1$, consider the word $\omega_k$, which is either of the form $\omega_k=a_1a_2\cdots a_p$ or $\omega_k=a_1a_2\cdots a_p/a_{p+1}$. Form a new word $\eta_k$ from the word $a_1a_2\cdots a_p$ by inserting the word $a_ia_i^{-1}$ after every position $i$ for which $\chi(a_1a_2\cdots a_i)\equiv \pm 1$ (mod $2^{k-n}$), $\chi(a_1a_2\cdots a_{i-1})\not\equiv 0$ (mod $2^{k-n}$) and $\chi(a_1a_2\cdots a_{i+1})\not\equiv 0$ (mod $2^{k-n}$). We then define the word $\tau_n$ by $\tau_n=\phi_n\circ \phi_{n+1}\circ\cdots \circ \phi_{k-1}(\eta_k)$ for sufficiently large $k$.
\end{definition}

\begin{remark}\label{CompleteReduce} If $(\omega_n)_n$ is of terminating type, then  $\left(\overline{(\omega_n)_n}\right)'=(\omega_n)'_n$ \cite[Lemma~6.4]{FZ}. In particular, we have the following commutative diagram of bijections:

\hspace{-18pt} \parbox{8in}{
\[
\xymatrix{ \overline{\overleftarrow{\mathcal G}} \ar@{.>}[dr]_{'}^{ bijections} & &  \overleftarrow{\mathcal G}\; \ar@{.>}_{``\overline{\;\;\;}"}[ll] \\
& \; {\mathcal G} \ar@{.>}_{``\overleftarrow{\;\;\;}"}[ur]}
\]}
\end{remark}

\subsection{Stable initial match} Some care must be taken in pinpointing the exact moment where two word sequences diverge, because seeing more detail allows for noticing more differences. We recall from \cite[Definition~4.15]{FZ}:

\begin{definition}[Stable initial match]
For $\omega_n,\xi_n\in \mathcal W_+$, we define $\omega_n\cap \xi_n\in \mathcal W_+$ as the maximal consecutive initial substring of the letters of the two words, including any letters that might come after the symbol ``/'', where  we separate the last letter of $\omega_n\cap \xi_n$ by the symbol ``/'' if it is so separated in the shorter of the two words.
For $(\omega_n)_n,(\xi_n)_n\in \mathcal W\mathcal S$, the projections $\tau_n=\phi_n\circ \phi_{n+1} \circ \cdots \circ \phi_{k-1}(\omega_k\cap \xi_k)$ are constant for sufficiently large $k>n$, and we define $(\omega_n)_n\Cap (\xi_n)_n=(\tau_n)_n$.
\end{definition}

\subsection{Dynamic word length}

Since edge-paths in \cite{FZ} are recorded as words of visited {\em vertices}, rather than traversed {\em edges}, the recursive word length function from \cite[Definition~4.14]{FZ} (which is based on \cite{MO}) needs to be adjusted slightly (a game interpretation is given in Remark~\ref{speed} below):

  \begin{definition}\label{weights}
  Let $(\omega_n)_n\in {\mathcal WS}$.
  \begin{itemize}

  \item[(a)] Given $\omega_1=a_1a_2\cdots a_p/\ast$, i.e.,  $\omega_1=a_1a_2\cdots a_p$ or $\omega_1=a_1a_2\cdots a_p/a_{p+1}$, we insert geometrically decreasing  weights: $(\frac{1}{2})a_1(\frac{1}{4})a_2(\frac{1}{8}) \cdots (\frac{1}{2^p}) a_p ( \frac{1}{2^{p+1}})$.

  \item[(b)] Given $\omega_n=b_1b_2\cdots b_{k-1}/\ast$ with weights $(\rho_1) b_1 (\rho_2) b_2 (\rho_3) \cdots (\rho_{k-1}) b_{k-1}(\rho_k)$, and $\omega_{n+1}=a_1a_2\cdots a_p/\ast$, we define $\omega^i_{n+1}=c_{i,1} c_{i,2}\cdots c_{i,m_i}$ ($1\leqslant i \leqslant k$)  as in Definition~\ref{phi} for $\phi_n(a_1a_2\cdots a_p)=b_1b_2\cdots b_{k-1}$, and insert the weights
     \[\left(\frac{\rho_1}{2}\right) c_{1,1} \left(\frac{\rho_1}{4}\right) c_{1,2} \left(\frac{\rho_1}{8}\right) \cdots \left(\frac{\rho_1}{2^{m_1}}\right) c_{1,m_1} \left(\frac{\rho_1}{2^{m_1+1}}\right), \]

        \[c_{2,1}\left(\frac{\rho_{1}}{2^{m_{1}+1}}+\frac{\rho_2}{2}\right) c_{2,2}\left(\frac{\rho_2}{4}\right) c_{2,3} \left(\frac{\rho_2}{8}\right) \cdots \left(\frac{\rho_2}{2^{m_2-1}}\right) c_{2,m_2} \left(\frac{\rho_2}{2^{m_2}}\right), \mbox{ and }\]

     \[c_{i,1}\left(\frac{\rho_{i-1}}{2^{m_{i-1}}}+\frac{\rho_i}{2}\right) c_{i,2}\left(\frac{\rho_i}{4}\right) c_{i,3} \left(\frac{\rho_i}{8}\right) \cdots \left(\frac{\rho_i}{2^{m_i-1}}\right) c_{i,m_i} \left(\frac{\rho_i}{2^{m_i}}\right) \mbox{ for $i>2$}.\]
\end{itemize}\vspace{5pt}
We let $|\omega_n|$ denote the sum of the weights so inserted into the word $\omega_n$. Then  $|\omega_1|>|\omega_2|> |\omega_3|>\cdots$. We define $\displaystyle \|(\omega_n)_n\|=\lim_{n\rightarrow\infty} |\omega_n|$.
  \end{definition}

\begin{remark}
 Let $(\omega_n)_n\in \mathcal W\mathcal S$. Suppose the word $\omega_n=a_1a_2\cdots a_p/\ast$ has weights $(\rho_1)a_1 (\rho_2) a_2 (\rho_3) \cdots  (\rho_p)a_p (\rho_{p+1})$. Then, for every $2\leqslant i \leqslant p+1$, we have $\rho_1=1/2^n$, $0<\rho_{i-1}/2^n\leqslant \rho_i\leqslant (3/4)^{n-1}/2$ and  $|\omega_n|-\rho_p-\rho_{p+1}<\|(w_m)_m\|<|\omega_n|$ \cite[Lemmas~6.34 and 6.35]{FZ}. In particular, the empty word $\omega_n=\emptyset\in {\mathcal W}_+$ has length $|\emptyset|=1/2^n$, so that the sequence $(\emptyset)_n\in {\mathcal WS}$ has length $\|(\emptyset)_n\|=0$.
\end{remark}

\begin{remark} \label{speed} {\sl We  think of the weights in Definition~\ref{weights} as time limits (in fractions of some unit of time) that we give to the player for moving the disks. (We assume that our idealized player is capable of moving any disk at any speed.)
 For $\omega_1$, two disks need to be moved alternately, Disk~1 and Disk~2. Every time one of these two disks are handled by the player, we allow less time to do so: $1/2, 1/4, 1/8,\dots$. The total allowed time for completing the movements of $\omega_1$ is given by $|\omega_1|$. For example, during $\omega_1=x^{+1}x^{+1}y^{-1}$, we handle disks on four occasions, adding up to a total time limit of $|\omega_1|=1/2+1/4+1/8+1/16=1-1/16$.

Playing the game according to  $\omega_{n+1}$, there are $n+2$ disks. Recursively, time limits for the disk movements according to $\omega_n$ have already been established. As discussed in Remark~\ref{backandforth} and Section \ref{visualization}, every handling of a Disk~$d$ with $d\in \{1,2,\dots,n+1\}$ is intermittently interrupted by the handling of Disk $n+2$, as recorded by $\omega^i_{n+1}$ in Section~\ref{formula}. If $\rho$ denotes the allotted time for the corresponding handling of Disk $d$ as part of the word $\omega_n$, we allow the following lengths of time to handle either Disk $d$ or Disk $n+2$ during $\omega^i_{n+1}$: $\rho/2, \rho/4, \rho/8,\dots$. The remaining fractional time is being added to the allowed time for handling the very next disk, Disk $d'$ of Remark~\ref{backandforth}. The total allowed time to finish $\omega_{n+1}$ equals $|\omega_{n+1}|$.

For example, consider the word $\omega_7= x^{+1} y^{-1} y^{+1}\mid x^{+1} y^{-1} x^{+1} y^{+1} x^{-1} y^{+1}\mid  y^{+1} x^{-1} y^{+1} y^{-1}\mid y^{-1} x^{-1} x^{+1} \cdots$, subdivided as $\omega_7^1, \omega_7^2, \omega_7^3, \dots$. Then $\omega_7$  maps to $\omega_6=\phi_6(\omega_7)=x^{+1}y^{+1} x^{-1}  \cdots$. If the allotted time limits for the disk movements for $\omega_6$ are $\rho_1, \rho_2, \rho_3, \rho_4, \dots$, then the time limits for $\omega_7$ are $\rho_1/2, \rho_1/4,\rho_1/8, \rho_1/16,$ $(\rho_1/16 +\rho_2/2), \rho_2/4, \rho_2/8, \rho_2/16,$ $\rho_2/32, \rho_2/64,(\rho_2/64+\rho_3/2), \rho_3/4, \rho_3/8,\rho_3/16,$\linebreak $(\rho_3/16+\rho_4/2), \rho_4/4, \rho_4/8, \cdots$, which need to be added  to obtain $|\omega_7|$.

 At the very end of every word $\omega_n$, we loose a little bit of time, since there is no next disk which could make use of the fractional carryover. The effect of this is that the more disks are on the board, the faster the player has to finish the game.}
\end{remark}

\subsection{The $\mathbb{R}$-tree.} In summary, the two-letter word calculus described herein, as illustrated by the corresponding movements of disks in our version of the Towers of Hanoi, gives an explicit description of the generalized Cayley graph for $\pi_1(\mathbb{M})$:

\begin{theorem}[Theorems~B--E of \cite{FZ}] \label{B-E}\hspace{1in}
\begin{itemize}
\item[(a)]  For word sequences $(\omega_n)_n, (\xi_n)_n\in \overleftarrow{\mathcal R}$, define \[\rho((\omega_n)_n,(\xi_n)_n)=\Big\|\overline{(\omega_n)_n}\Big\|+\Big\|\overline{(\xi_n)_n}\Big\|-2\Big\|\overline{(\omega_n)_n}\Cap \overline{(\xi_n)_n}\Big\|.\] Then $\rho$ is a pseudo metric on $\overleftarrow{\mathcal R}$ with $\rho((\omega_n)_n,(\xi_n)_n)=0 \Leftrightarrow (\omega_n)_n\doteq (\xi_n)_n$.  Moreover, the resulting metric space $(\dot{\overleftarrow{\mathcal R}},\rho)$ is an $\mathbb{R}$-tree.\vspace{3pt}

\item[(b)] The group $\overleftarrow{\mathcal G}\cong \pi_1(\mathbb{M})$ acts freely and by homeomorphisms on the $\mathbb{R}$-tree $\dot{\overleftarrow{\mathcal R}}$ via its natural action $(\omega_n)_n.(\xi_n)_n=\overleftarrow{(\omega_n \xi_n)'_n}$ with quotient $\dot{\overleftarrow{\mathcal R}}/\overleftarrow{\mathcal G}\approx\mathbb{M}$.\vspace{3pt}

    \noindent{\em (}Note that the action cannot possibly be by isometries; see Example~\ref{HE2} below.{\em )}

\item[(c)] For word sequences $(\omega_n)_n,(\xi_n)_n\in \overleftarrow{\mathcal G}\cong \pi_1(X,x)$, the arc of the $\mathbb{R}$-tree $\dot{\overleftarrow{\mathcal R}}$ from $(\omega_n)_n$ to $(\xi_n)_n$ naturally spells out
    the word sequence  $\overline{(\omega_n)_n^{-1}\ast (\xi_n)_n}$,
      when projecting the arc to its edge-paths: $\dot{\overleftarrow{\mathcal R}}\rightarrow \mathbb{M}\rightarrow X_n$.
\end{itemize}
\end{theorem}

\section{An embedding of the Hawaiian Earring}\label{HEemb}

\begin{example}  \label{HE1} Here is a sequence which is {\em not} locally eventually constant.
Consider the word \[\ell_{(k)}=x^{2^{k-1}}y^{-1}x^{2-2^k}y^{-1}x^{2^{k-1}}\in \Omega_n \;\;(1\leqslant k\leqslant n).\] (For fixed $k$, the same word $\ell_{(k)}$ corresponds to an edge-path in $X_n$ for each $n\geqslant k$.)
Then $\phi_{n-1}(\ell_{(k)})=\ell_{(k-1)}$ for $2\leqslant k \leqslant n$ and  $\phi_{n-1}(\ell_{(1)})=$  empty word. (See Remark~\ref{movements}.)
Following \cite[p.~185]{G}, consider \[[n,k]=\ell_{(n)}\ell_{(k)}\ell_{(n)}^{-1}\ell_{(k)}^{-1},\] put $\omega_1=\emptyset\in \Omega_1$ and
\[\omega_n=[n,n-1][n,n-2]\cdots [n,1]\in \Omega_n \mbox{ for } n>1.\]
(Here, we define $(a_1a_2\cdots a_m)^{-1}=a_m^{-1}a_{m-1}^{-1}\cdots a_1^{-1}$.) Then  $(g_n)_n:=(\omega_n')_n\in G$, because $\phi_{n-1}'(\omega_n')=\phi_{n-1}(\omega_n)'=(\omega_{n-1}\ell_{(n-1)}\ell_{(n-1)}^{-1})'=\omega_{n-1}'$.

In order to see why $(g_n)_n$ is not locally eventually constant, observe that (by Definition~\ref{phi}) if $\phi_{n-1}$ is applied to any word of $\Omega_n$ containing a subword of length $2^n-1$ with equal exponents, it produces a word of $\Omega_{n-1}$ containing a corresponding subword of length at least $2^{n-1}-1$ with equal exponents. Consequently, since the central $2^n$ letters of each of the $2n-2$ subwords $\ell^{\pm 1}_{(n)}$ of $\omega_n$ remain after we carry out all possible cancellations in $\omega_n$, the word $\phi_1\circ \phi_2\circ \cdots \circ \phi_{n-1}(g_n)$ contains at least $2n-2$ letters.
(In fact, a more detailed analysis would show that $\phi_1\circ \phi_2\circ \cdots \circ \phi_{n-1}(g_n)=\left(\ell_{(1)}\ell_{(1)}^{-1}\right)^{n-1}$.)
\end{example}

\begin{remark}\label{movements}
{\sl
The word $\ell_{(k)}\in \Omega_n$ ($1\leqslant k\leqslant n$) in Example~\ref{HE1} describes the following movements of the $n+1$ disks: Place the entire stack on Peg 0, play the top stack consisting of the smallest $k-1$ disks to Peg $P$, where $P\in\{1,2\}$ and $P\equiv n-k$ (mod 2), turn over the $k^{\mbox{th}}$ smallest disk, play the top stack back to Peg 0, move the entire stack to Peg 2, play the top stack consisting of the smallest $k-1$ disks to Peg $P$, where $P\in\{0,1\}$ and $P\equiv n-k$ (mod 2), turn over the $k^{\mbox{th}}$ smallest disk, play the top stack back to Peg 2, and finally pick up the entire stack. (Note that the first (resp. second) time around, the top stack reassembles at the unique Peg $P$ that would allow the $k^{\mbox{th}}$ smallest disk to move in (resp. against) its natural running direction.)}
\end{remark}

\begin{example}[cf. Example~4.15 of \cite {FZ2}]\label{HE2} There is no $\mathbb{R}$-tree metric on  $\dot{\overleftarrow{\mathcal R}}$ which renders the action of $\overleftarrow{\mathcal G}$ as isometries. Suppose, to the contrary, that there is a metric $\kappa: \dot{\overleftarrow{\mathcal R}}\times \dot{\overleftarrow{\mathcal R}}\rightarrow [0,\infty)$ such that $(\dot{\overleftarrow{\mathcal R}},\kappa)$ is an $\mathbb{R}$-tree and such that the action of $\overleftarrow{\mathcal G}$ on $(\dot{\overleftarrow{\mathcal R}},\kappa)$ is by isometries. Using the words $\ell_{(k)}$
from Example~\ref{HE1}, define \vspace{5pt}

\hspace{1in}  $L_1=(\ell_{(1)}, \ell_{(2)}, \ell_{(3)}, \ell_{(4)}, \ell_{(5)}, \cdots)\in \overleftarrow{\mathcal G}$

\hspace{1in} $L_2=( \hspace{5pt} \emptyset \hspace{5pt}, \ell_{(1)}, \ell_{(2)}, \ell_{(3)},\ell_{(4)}, \cdots)\in \overleftarrow{\mathcal G}$

\hspace{1in} $L_3=(\hspace{5pt} \emptyset \hspace{5pt}, \hspace{5pt} \emptyset \hspace{5pt}, \ell_{(1)}, \ell_{(2)}, \ell_{(3)}, \cdots)\in \overleftarrow{\mathcal G}$

\hspace{1in} $L_4=(\hspace{5pt} \emptyset \hspace{5pt}, \hspace{5pt} \emptyset \hspace{5pt}, \hspace{5pt} \emptyset \hspace{5pt}, \ell_{(1)}, \ell_{(2)},\cdots)\in \overleftarrow{\mathcal G}$

\hspace{1.225in} $\vdots$
\vspace{5pt}

\noindent  Choose $n_i\in \mathbb{N}$ with $n_i>1/\kappa((\emptyset)_n,L_i)$. Then $L=L_1^{n_1}L_2^{n_2}L_3^{n_3}\cdots \in \overleftarrow{\mathcal G}\subseteq \dot{\overleftarrow{\mathcal R}}$. By Theorem~\ref{B-E}(c) and Remark~\ref{CompleteReduce}, the arc in $\dot{\overleftarrow{\mathcal R}}$ from $(\emptyset)_n$ to $L$ contains the points $L_1^{n_1}L_2^{n_2}L_3^{n_3}\cdots L_i^j\in \overleftarrow{\mathcal G}$ for all $i\in \mathbb{N}$ and $1\leqslant j \leqslant n_i$.
Therefore,  assuming the action of $\overleftarrow{\mathcal G}$ on $(\dot{\overleftarrow{\mathcal R}},\kappa)$ is by isometries, the length of this arc equals $\displaystyle \sum_{i=1}^\infty n_i\kappa((\emptyset)_n,L_i)=\infty$. However, the length of an arc in an $\mathbb{R}$-tree  equals the distance between its endpoints.
\end{example}

\begin{remark}
We mention in passing that $\rho((\emptyset)_n,L_i)=(11/12)(1/2)^{i-1}$. To verify this, it suffices to show that $||L_1||=11/12$. This, in turn, follows from three observations: (i) By Definition~\ref{weights}(a), $|\ell_{(1)}|=1-1/32$; (ii) If the subwords $\omega^1_{n+1},\omega^2_{n+1}, \dots, \omega^{k}_{n+1}$  are as in Definitions~\ref{phi} \& \ref{weights}(b) for  $\phi_{n}(\ell_{(n+1)})=\ell_{(n)}$, then $k=2^{n+1}+1$, the word $\omega^{k-1}_{n+1}=c_{k-1,1}c_{k-1,2}$ has two letters and  the word $\omega^k_{n+1}=c_{k,1}$ has one letter; (iii) Consequently, by Definition~\ref{weights}(b), the weight which follows $c_{k,1}$, call it $\varrho_{n+1}=\lambda_{n+1}/2^{2n+5}$, satisfies the recursion $\lambda_1=1$ and $\lambda_{n+1}=2\lambda_n+2$, so that $\lambda_{n+1}=3\cdot 2^n-2$. Hence, $||L_1||= 1-1/32 - (1/2)\varrho_1 - (1/2)\varrho_2 - (1/2)\varrho_3 - \cdots = 11/12$.
\end{remark}

\begin{example}[A Hawaiian Earring]\label{HE3}
Let $L_i$ be as in Example~\ref{HE2}.
Let $C_i\subseteq \dot{\overleftarrow{\mathcal R}}/\overleftarrow{\mathcal G}\cong\mathbb{M}$ be the image of the arc in $\dot{\overleftarrow{\mathcal R}}$ from  $(\emptyset)_n$ to $L_i$.  Then each $C_i$ is a simple closed curve and $C_{i+1}$ shares an arc with $C_i$ that contains the base point of $\mathbb{M}$. The subspace $\bigcup_{i\in\mathbb{N}} C_i$ of $\mathbb{M}$ is clearly homotopy equivalent to the Hawaiian Earring $\mathbb{H}=\{(x,y)\in \mathbb{R}^2\mid x^2+(y-1/i)^2=(1/i)^2 \mbox{ for some } i\in \mathbb{N}\}$.
Let $Y_n$ be the image of $\bigcup_{i=1}^n C_i$ in $X_n$. Put $\alpha_{n,i}=\ell_{(n+1-i)}$. Then $\phi_{n-1}(\alpha_{n,i})=\alpha_{n-1,i}$ for $1\leqslant i\leqslant n-1$ and $\phi_{n-1}(\alpha_{n,n})=$ empty word. Let ${\mathcal L}_n$ be the set of all finite words over the alphabet  $\{\alpha_{n,1}^{\pm 1}, \alpha_{n,2}^{\pm 1}, \dots, \alpha_{n,n}^{\pm 1}\}$ and let $\varepsilon:{\mathcal L}_n\rightarrow \Omega_n$ be the natural letter substitution.
Considering
${\mathcal H}=\{(g_n)_n\in {\mathcal G}\mid \mbox{for every $n$, there is an } \omega_n\in \varepsilon({\mathcal L}_n) \mbox{ with } g_n=\omega_n'\}$ and
applying \cite[Theorem~A]{FZ} to \[\bigcup_{i\in \mathbb{N}} C_i=\lim_{\longleftarrow} \left(Y_1\stackrel{f_1|_{Y_2}}{\longleftarrow} Y_2 \stackrel{f_2|_{Y_3}}{\longleftarrow} Y_3 \stackrel{f_3|_{Y_4}}{\longleftarrow} \cdots\right),\] we see that the subgroup $\overleftarrow{\mathcal H}$ of $\overleftarrow{\mathcal G}$ is isomorphic to $\pi_1(\mathbb{H})$.

This is seen to be equivalent to the known word sequence description of $\pi_1(\mathbb{H})$ as detailed, for example,  in \cite{Z}, upon formally identifying  all $\alpha_{n,i}$ ($n\in \mathbb{N})$ to $\alpha_i$,\linebreak interpreting  $\phi_{n-1}$ as the function ${\mathcal L}_{n}\rightarrow {\mathcal L}_{n-1} $ which deletes all occurrences of $\alpha_n$,\linebreak and considering the stabilizations of the locally eventually constant sequences in the limit of the reduced inverse sequence of free groups \[F(\alpha_1)\stackrel{\phi'_1}{\leftarrow} F(\alpha_1,\alpha_2) \stackrel{\phi'_2}{\leftarrow} F(\alpha_1,\alpha_2,\alpha_3) \stackrel{\phi'_3}{\leftarrow}\cdots,\] under the group operation of termwise concatenation, followed by reduction and restabilization.
\end{example}

\begin{remark}
With a small modification of the definition for $L_i$ in Example~\ref{HE2} (leaving the distance $\rho((\emptyset)_n,L_i)$ unchanged) the corresponding simple closed curves $C_i$ of Example~\ref{HE3}, pairwise, only share the base point of
$\mathbb{M}$, so that $\bigcup_{i\in\mathbb{N}} C_i$ is {\em homeomorphic} to $\mathbb{H}$. {\sl Specifically, where previously each $L_i$ was defined as a game sequence in which only the $(i+1)^{\mbox{st}}$ disk is turned (twice), one would instead define $L_i$ as a game sequence in which every $k\cdot(i+1)^{\mbox{st}}$ disk (for $k=1,2,3,\dots$) is turned exactly twice, once from white to black as soon as possible, and once from black to white at the last opportunity.} We leave it to the reader to formulate this in terms of a word sequence.
\end{remark}

\noindent {\bf Acknowledgements}.  This work was partially supported by a grant from the Simons Foundation (\#245042 to Hanspeter Fischer).

\end{document}